\documentclass{article}

\usepackage{graphicx}

\usepackage{authblk}

\usepackage{tikz}
\usetikzlibrary{arrows,calc}

\tikzstyle{arc}=[->,shorten <=3pt, shorten >=3pt,
                 >=stealth, line width=1.1pt]
\tikzstyle{edge}=[shorten <=2pt, shorten >=2pt,
                  >=stealth, line width=1.1pt]
\tikzstyle{vertex}=[circle, fill=white, draw,
                    minimum size=5pt,
                    inner sep=0pt, outer sep=0pt]

\usepackage{float}

\usepackage{xcolor}

\usepackage{amsmath}
\usepackage{amssymb}

\usepackage{xfrac}

\usepackage{amsthm}

\usepackage[capitalize]{cleveref}

\newtheorem{theorem}{Theorem}
\newtheorem{lemma}[theorem]{Lemma}
\newtheorem{corollary}[theorem]{Corollary}
\newtheorem{remark}[theorem]{Remark}
\newtheorem{observation}[theorem]{Observation}
\newtheorem{proposition}[theorem]{Proposition}
\newtheorem{problem}[theorem]{Problem}
\newtheorem{question}[theorem]{Question}
\newtheorem{claim}{Claim}

\title{Describing hereditary properties by forbidden circular orderings%
\thanks{The authors gratefully acknowledge support from NSERC Canada,
SEP-CONACYT A1-S-8397, CONACYT FORDECYT-PRONACES/39570/2020 and DGAPA-PAPIIT
IA104521 grants}}

\author[1]{Santiago~Guzm\'an-Pro\thanks{sanguzpro@ciencias.unam.mx}}
\author[2]{Pavol~Hell\thanks{pavol@sfu.ca}}
\author[1]{C\'esar~Hern\'andez-Cruz\thanks{chc@ciencias.unam.mx}}

\affil[1]{Facultad de Ciencias\\
  Universidad Nacional Aut\'onoma de M\'exico\\
  Av. Universidad 3000, Circuito Exterior S/N\\
  C.P. 04510, Ciudad Universitaria, CDMX, M\'exico}
\affil[2]{School of Computing Science\\
  Simon Fraser University\\
  8888 University Drive\\
  Burnaby, B.C., Canada, V5A 1S6}

\begin{document}
\date{}

\maketitle
\begin{abstract}
  Each hereditary property can be characterized by its set of minimal
  obstructions; these sets are often unknown, or known but infinite. By allowing
  extra structure it is sometimes possible to describe such properties by a
  finite set of forbidden objects. This has been studied most intensely when the
  extra structure is a linear ordering of the vertex set. For instance, it is
  known that a graph G is $k$-colourable if and only if $V(G)$ admits a linear
  ordering $\le$ with no vertices $v_1 \le \cdots \le v_{k+1}$ such that $v_i
  v_{i+1} \in E(G)$ for every $i \in \{ 1, \dots, k \}$. In this paper, we study
  such characterizations when the extra structure is a circular ordering of the
  vertex set. We show that the classes that can be described by finitely many
  forbidden circularly ordered graphs include forests, circular-arc graphs, and
  graphs with circular chromatic number less than $k$. In fact, every
  description by finitely many forbidden circularly ordered graphs can be
  translated to a description by finitely many forbidden linearly ordered
  graphs. Nevertheless, our observations underscore the fact that in many cases
  the circular order descriptions are nicer and more natural.
\end{abstract}

\section{Introduction}
\label{sec:Introduction}

We follow \cite{bondy2008} for terminology and notation not defined here, and we
consider simple finite graphs; when needed, we will work with loopless oriented
graphs as well. A \textit{hereditary property} is a class of graphs
$\mathcal{P}$ such that whenever $G \in \mathcal{P}$ and $H$ is an induced
subgraph of $G$, then $H \in \mathcal{P}$. A \textit{minimal obstruction} to a
hereditary property $\mathcal{P}$ is a graph $G$ that does not belong to
$\mathcal{P}$ but every proper induced subgraph does. A natural way to
characterize or define a hereditary property is by exhibiting its set of minimal
obstructions. For instance, bipartite graphs are characterized as those graphs
with no induced odd cycles, while the class of evenhole-free graphs is defined
as the class of graphs that contain no even cycle as an induced subgraph.
Unfortunately exhibiting the set of minimal obstructions might be a highly
complex task; as of today, the sets of minimal obstructions to the class of
$k$-colourable graphs is unknown for every positive integer $k$ greater than
$2$.

A \textit{linearly ordered graph} $(G,\le)$ is a graph $G$ together with a
linear ordering $\le$ of its vertex set. Given two linearly ordered graphs, $(G,
\le_G)$ and $(H, \le_H)$, we say that $(H, \le_H)$ is a \textit{linearly ordered
subgraph} of $(G, \le_G)$ if $H$ is a subgraph of $G$, and $\le_H$ is the
restriction of $\le_G$ to $V(H)$; if additionally $H$ is an induced subgraph of
$G$, we say that $(H, \le_H)$ is an \textit{induced linearly ordered subgraph}
of $(G, \le_G)$. Consider a set of linearly ordered graphs $F$. An
\textit{$F$-free linear ordering} of a graph $G$ is a linear ordering $\le$ of
$V(G)$ such that none of the linearly ordered graphs in $F$ is an induced
linearly ordered subgraph of $(G,\le)$. Given a linear order $\le$ of some set
$X$, we define the \textit{dual} $\le^\ast$ of $\le$, by letting $x \le^\ast y$
if and only if $y \le x$.

In 1990 Damaschke \cite{damaschkeTCGT1990} proposed to study characterizations
of hereditary properties $\mathcal{P}$ by exhibiting a finite set of linearly
ordered graphs $F$ such that $\mathcal{P}$ is the class of graphs that admit an
$F$-free linear ordering. He observed that, for instance, chordal graphs,
bipartite graphs and interval graphs are characterized by a forbidden set of
linearly ordered graphs on three vertices; also in \cite{damaschkeTCGT1990} he
asked if the class of circular-arc graphs can be described by finitely many
forbidden linearly ordered graphs. We will see that we can reinterpret a (known)
characterization of circular-arc graphs in our context to obtain a positive
answer to Damaschke's question.

Around 2014, Hell, Mohar and Rafiey \cite{hellESA2014} showed that for every set
$F$ of linearly ordered graphs on three vertices, the class of graphs that admit
an $F$-free linear ordering can be recognized in polynomial time. Recently,
Habib and Feuilloley published a thorough survey \cite{feuilloleyJDM} on the
subject, where they characterized all hereditary properties defined by forbidden
linear ordering on three vertices. Moreover, they showed that all of these
classes (except for two of them) can be recognized in linear time. In their
work, Habib and Feulilloley, stated that an obvious next step is to  study graph
properties described by forbidden linear orderings on more vertices.  All of our
results can be translated to this context.

In this work we start the study of circularly ordered graphs, in an attempt to
obtain a development parallel to the one described in the above paragraphs for
linearly ordered graphs.   We also present an interesting result relating strict
upper bounds on the circular chromatic number of graphs to certain forbidden
circular orderings.

This work is structured as follows. For the rest of this section we mention some
definitions and results on circular chromatic number of graphs that we will use
in Section~\ref{sec:circhrom}. In Section~\ref{sec:Basic} we introduce basic
definitions and notation to work with circularly ordered graphs, and we observe
that circular-arc graphs and outerplanar graphs can be described by finitely
many forbidden circularly ordered graphs. In Section~\ref{sec:ordpat} we study
how forbidden circular orderings and forbidden linear orderings are related;
moreover, we exhibit some properties expressible by finitely many forbidden
circularly ordered graphs. In Section~\ref{sec:circhrom} we explore a nice
relation between forbidden circularly ordered graphs and integer circular
chromatic numbers.  Finally, in Section~\ref{sec:complexity} we discuss some
computational aspects of finding admissible circular orderings of a given input
graph.

Recall that a \textit{homomorphism} between a pair of relational structures, $G$
and $H$, is a function $\varphi \colon V(G) \to V(H)$ that preserves all
relations. If such a function exists we write $G \to H$. In particular,
homomorphisms between graphs are functions that preserve adjacencies, so a graph
$G$ is $k$-colourable if and only if $G \to K_k$. (We will later work with
homomorphisms of linearly and circularly ordered graphs as well.) A bijective
homomorphism such that its inverse is also a homomorsphism is an
\textit{isomorphism}.

Given a pair of positive integers $p$ and $q$, $q \le p$, the \textit{rational
complete graph} $K_{\sfrac{p}{q}}$ has vertices $\{0,1,\dots, p-1\}$ and there
is an edge $ij$ if and only if the circular distance between $i$ and $j$ is at
least $q$. In particular, if $p < 2q$ then $K_{\sfrac{p}{q}}$ is the empty graph
on $p$ vertices, and $K_{\sfrac{p}{1}}$ is the complete graph on $p$ vertices.
These graphs have a nice monotonic property with respect to the
natural ordering of rational numbers and graph homomorphisms.

\begin{theorem}\cite{hell2004}
\label{thm:pqhoms}
Consider a pair of positive integers $p$ and $q$ such that $\sfrac{p}{q} \ge 2$.
Then $K_{\sfrac{p}{q}} \to K_{\sfrac{p'}{q'}}$ if and only if $\sfrac{p}{q} \le
\sfrac{p'}{q'}$.
\end{theorem}

A graph $G$ is \textit{$(\sfrac{p}{q})$-colourable} if $G \to K_{\sfrac{p}{q}}$.
The \textit{circular chromatic number} of a graph $G$, denoted by $\chi_c(G)$,
is defined as
\[
  \chi_c(G) = \inf\{\sfrac{p}{q}\colon p\le n, G \to K_{\sfrac{p}{q}}\}
\]
where $n  = |V(G)|$. It turns out that this infimum is always reached.

\begin{proposition}\label{prop:min}\cite{hell2004}
For a graph $G$ on $n$ vertices, we have
\[ \chi_c(G) = \min\{\sfrac{p}{q}\colon p\le n, G \to K_{\sfrac{p}{q}}\}.
\]
\end{proposition}

As a nice consequence of these results, for every graph $G$ the inequalities
$\chi(G) -1 < \chi_c(G) \le \chi(G)$ hold.

There are several interpretations of the circular chromatic number of graph, the
following one will be useful for our work. Before stating it, recall that given
a (possibly closed) walk $W = v_1v_2\cdots v_n$ in an oriented graph $G'$, an
arc $(v_i,v_{i+1})$ is a \textit{forward arc} of $W$ while an arc
$(v_{i+1},v_i)$ is a \textit{backward arc} of $W$. We denote by $W^+$ ($W^-$)
the set of forward (backward) arcs of $W$.

\begin{theorem}\cite{hell2004}
\label{thm:orientationsCirc}
If $G$ is a forest, then $\chi_c(G) = 2$. Otherwise, $\chi_c(G)$ is the minimum
over all acyclic orientations $G'$ of $G$, of the maximum, over all cycles $C$
of $G$, of
\[
  1+ \frac{|C^+|}{|C^-|}.
\]
\end{theorem}

%%%%%%%%%%%%%%%%%%%%%%%%%%%%%%%%%%%%
%%%%%%%%%%%%%%%%%%%%%%%%%%%%%%%%%%%%
\section{Circularly ordered graphs}
\label{sec:Basic}

A \textit{circular ordering} of a set $X$ is a ternary relation $C \subseteq
X^3$ such that for any four elements $x,y,z,w\in X$ the following statements
hold:
\begin{itemize}
	\item if $(x,y,z) \in C$ then $(y,z,x) \in C$,

	\item if $(x,y,z) \in C$ then $(x,z,y) \notin C$,

	\item if $(x,y,z) \in C$  and $(x,z,w) \in C$, then $(x,y,w) \in C$, and

	\item either  $(x,y,z) \in C$ or  $(x,z,y) \in C$.
\end{itemize}

A \textit{circularly ordered graph} $G$ is an ordered pair $G = (U_G, C)$, where
$U_G$ is a graph (the \textit{underlying graph} of $G$) and $C$ is a circular
ordering of $V(U_G)$. We will often abuse nomenclature and say that a
circularly ordered graph $(G, C)$ is a circular ordering of the graph $G$.
Notice that each graph on two or three vertices defines a unique circularly
ordered graph; in Figure~\ref{fig:allcirc} we depict all circularly ordered
graphs on four vertices.

Since circularly ordered graphs are relational structures, we use the
isomorphism definition for relational structures to compare circularly ordered
graphs. In particular, the underlying graphs of two isomorphic circularly
ordered graphs are isomorphic (via graph isomorphism). We say that $H$ is an
\textit{(induced) circularly ordered subgraph} of $G$ if $U_H$ is an (induced)
subgraph of $U_G$ and $C_H$ is the restriction of $C_G$ to $V(U_H)$.  We also
say that $H$ is a \textit{spanning circularly ordered supergraph} of $G$ if $G$
is a circularly ordered subgraph of $H$, and $V(U_H) = V(U_G)$. If a circularly
ordered graph isomorphic to $H$ is an induced circularly ordered subgraph of
$G$, we will say that $G$ \textit{contains} (an induced copy of) $H$.  For a set
$\mathcal{F}$ of circularly ordered subgraphs, we say that a circularly ordered
graph $G$ \textit{avoids} $\mathcal{F}$, or is \textit{$\mathcal{F}$-free}, if
$G$ does not contain any of the circularly ordered subgraphs in $\mathcal{F}$;
if $\mathcal{F} = \{ F \}$, we will abuse notation and say that $G$ avoids $F$
or is $F$-free.   A graph $G$ \textit{admits an $\mathcal{F}$-free circular
ordering} if there exists an $\mathcal{F}$-free circularly ordered graph
$(G,C)$.

Rather than representing a circular ordering by the ternary relation itself, we
will choose two simpler representations. Given a linear order $\le$ of a set
$X$, the \textit{circular closure} of $\le$ is a circular ordering $c(\le)$
defined as follows. For every $x, y$ and $z$ in $X$ such that $x \le y \le z$,
we have $(x,y,z) \in c(\le)$, and then, take the cyclic closure of these
triples, i.e., $(x,y,z) \in c(\le)$ if either $(y,z,x) \in c(\le)$ or $(z,x,y)
\in c(\le)$. Conversely, consider a circular ordering $C$ and any element $x \in
X$. Define $\le_x$ as $x\le_x y$ for any $y\in X$, and $y \le_x z$ if $(x,y,z)
\in C$. It is not hard to observe that $c(\le_x) = C$ for any $x \in X$. So we
can always describe a circular ordering of $X$ as the circular closure of a
linear ordering on $X$.

\begin{remark}
Consider a pair of linearly ordered graphs $(G,\le_G)$ and $(H,\le_H)$, and let
$(G,C_G)$ and $(H,C_H)$ be a pair of circularly ordered graphs. Then, the
following statements hold:
\begin{itemize}
	\item if $(H,\le_H)$ is a linearly ordered subgraph of $(G,\le_G)$, then
    $(H,c(\le_H))$ is a circularly ordered subgraph of $(G,c(\le_G))$, and

	\item if $(H,C_H)$ is a circularly ordered subgraph of $(G,C_G)$ then, for
    any $u \in V_G$ there is a vertex $v \in H$ such that $(H,\le_v)$ is a
    linearly ordered subgraph of $(G,\le_u)$.
\end{itemize}
Moreover, analogous statements when induced linearly ordered subgraphs and
induced circularly subgraphs are considered also hold.
\end{remark}

Consider the unit circle $S^1 \subseteq \mathbb{R}^2$ and a finite set $X$. Let
$f\colon X\to S^1$ be an injective function. Consider the ternary relation
$C(f)$ on $X$ defined by the ordered triples $(x,y,z)$ such that when traversing
$S^1$ in a clockwise direction starting in $f(x)$ we see $f(y)$ before $f(z)$.
It is not hard to convince ourselves that $C(f)$ is a circular ordering of $X$.
Conversely, let $C_X$ be any circular ordering on $X$, choose $x_1$ in $X$ so we
have $C_X = c(\le_{x_1})$, where $\le_{x_1}$ is the linear order $x_1 \le_{x_1}
x_2 \le_{x_1} \cdots \le_{x_1} x_n$. Define the function $f\colon X\to S^1$ by
$f(x_k) = (\cos \frac{k2\pi}{n}, -\sin \frac{k2\pi}{n})$. Clearly then, $f\colon
X\to S^1$ is an injective function of $X$ into the unit circle and $C(f) = C_X$.
This representation is specially useful when picturing a circular ordering. Due
to the arguments in these paragraphs we will refer to circular orderings as
\textit{circular arrangements} as well. Similarly, we will use the verb
\textit{arrange} to mean that we are constructing or defining a circular
ordering for a set (usually the vertex set of a graph).

%%%%%%%%%%%%%%%%%%%%%
%%%%%%%%%%%%%%%%%%%%%
%%%%----------- Figura con todos los de 4
%%%%%%%%%%%%%%%%%%%%%
%%%%%%%%%%%%%%%%%%%%%

\begin{figure}[ht!]
\centering
\begin{tikzpicture}[scale=0.6]

%%%%%% RENGLON 1

\begin{scope}[xshift=-9cm, yshift=10.5cm, scale=0.5] %%%K_4
\node [vertex, label=right:{$v_1$}] (1) at (2,0){};
\node [vertex, label=below:{$v_2$}] (2) at (0,-2){};
\node [vertex, label=left:{$v_3$}] (3) at (-2,0){};
\node [vertex, label=above:{$v_4$}] (4) at (0,2){};

\draw[edge]    (1)  edge  (2);
\draw[edge]    (1)  edge  (3);
\draw[edge]    (1)  edge  (4);
\draw[edge]    (2)  edge  (3);
\draw[edge]    (4)  edge  (2);
\draw[edge]    (4)  edge  (3);

\node (L) at (2.5,-2.6){$(a)$};
\end{scope}

\begin{scope}[xshift=-3cm, yshift=10.5cm, scale=0.5] %%%I_4
\node [vertex, label=right:{$v_1$}] (1) at (2,0){};
\node [vertex, label=below:{$v_2$}] (2) at (0,-2){};
\node [vertex, label=left:{$v_3$}] (3) at (-2,0){};
\node [vertex, label=above:{$v_4$}] (4) at (0,2){};

\node (L) at (2.5,-2.6){$(b)$};
\end{scope}

\begin{scope}[xshift=3cm, yshift=10.5cm, scale=0.5] %%%Claw
\node [vertex, label=right:{$v_1$}] (1) at (2,0){};
\node [vertex, label=below:{$v_2$}] (2) at (0,-2){};
\node [vertex, label=left:{$v_3$}] (3) at (-2,0){};
\node [vertex, label=above:{$v_4$}] (4) at (0,2){};

\draw[edge]    (1)  edge  (2);
\draw[edge]    (1)  edge  (3);
\draw[edge]    (1)  edge  (4);

\node (L) at (2.5,-2.6){$(c)$};
\end{scope}

\begin{scope}[xshift=9cm, yshift=10.5cm, scale=0.5]%Triangulito
\node [vertex, label=right:{$v_1$}] (1) at (2,0){};
\node [vertex, label=below:{$v_2$}] (2) at (0,-2){};
\node [vertex, label=left:{$v_3$}] (3) at (-2,0){};
\node [vertex, label=above:{$v_4$}] (4) at (0,2){};

\draw[edge]    (1)  edge  (4);
\draw[edge]    (1)  edge  (3);
\draw[edge]    (4)  edge  (3);

\node (L) at (2.5,-2.6){$(d)$};
\end{scope}

%%%%%%%%%%%% RENGLON 2

\begin{scope}[xshift=-9cm, yshift=6.4cm, scale=0.5] %%%2K_2 parallel
\node [vertex, label=right:{$v_1$}] (1) at (2,0){};
\node [vertex, label=below:{$v_2$}] (2) at (0,-2){};
\node [vertex, label=left:{$v_3$}] (3) at (-2,0){};
\node [vertex, label=above:{$v_4$}] (4) at (0,2){};

\draw[edge]    (1)  edge  (2);
\draw[edge]    (4)  edge  (3);

\node (L) at (2.5,-2.6){$(e)$};
\end{scope}

\begin{scope}[xshift=-3cm, yshift=6.4cm, scale=0.5] %%%2K_2 crossed
\node [vertex, label=right:{$v_1$}] (1) at (2,0){};
\node [vertex, label=below:{$v_2$}] (2) at (0,-2){};
\node [vertex, label=left:{$v_3$}] (3) at (-2,0){};
\node [vertex, label=above:{$v_4$}] (4) at (0,2){};

\draw[edge]    (4)  edge  (2);
\draw[edge]    (1)  edge  (3);

\node (L) at (2.5,-2.6){$(f)$};
\end{scope}

\begin{scope}[xshift=3cm, yshift=6.4cm, scale=0.5] %%%C_4 simple
\node [vertex, label=right:{$v_1$}] (1) at (2,0){};
\node [vertex, label=below:{$v_2$}] (2) at (0,-2){};
\node [vertex, label=left:{$v_3$}] (3) at (-2,0){};
\node [vertex, label=above:{$v_4$}] (4) at (0,2){};

\draw[edge]    (1)  edge  (2);
\draw[edge]    (2)  edge  (3);
\draw[edge]    (3)  edge  (4);
\draw[edge]    (1)  edge  (4);

\node (L) at (2.5,-2.6){$(g)$};
\end{scope}

\begin{scope}[xshift=9cm, yshift=6.4cm, scale=0.5] %%%C_4 crossed
\node [vertex, label=right:{$v_1$}] (1) at (2,0){};
\node [vertex, label=below:{$v_2$}] (2) at (0,-2){};
\node [vertex, label=left:{$v_3$}] (3) at (-2,0){};
\node [vertex, label=above:{$v_4$}] (4) at (0,2){};

\draw[edge]    (4)  edge  (3);
\draw[edge]    (1)  edge  (3);
\draw[edge]    (2)  edge  (1);
\draw[edge]    (2)  edge  (4);

\node (L) at (2.5,-2.6){$(h)$};
\end{scope}

%%%%%%%%%%%% RENGLON 3

\begin{scope}[xshift=-9cm, yshift=2.3cm, scale=0.5] %%%K_2 simple
\node [vertex, label=right:{$v_1$}] (1) at (2,0){};
\node [vertex, label=below:{$v_2$}] (2) at (0,-2){};
\node [vertex, label=left:{$v_3$}] (3) at (-2,0){};
\node [vertex, label=above:{$v_4$}] (4) at (0,2){};

\draw[edge]    (4)  edge  (3);

\node (L) at (2.5,-2.6){$(i)$};
\end{scope}

\begin{scope}[xshift=-3cm, yshift=2.3cm, scale=0.5] %%%K_2 slice
\node [vertex, label=right:{$v_1$}] (1) at (2,0){};
\node [vertex, label=below:{$v_2$}] (2) at (0,-2){};
\node [vertex, label=left:{$v_3$}] (3) at (-2,0){};
\node [vertex, label=above:{$v_4$}] (4) at (0,2){};

\draw[edge]    (1)  edge  (3);

\node (L) at (2.5,-2.6){$(j)$};
\end{scope}

\begin{scope}[xshift=3cm, yshift=2.3cm, scale=0.5] %%%simple
\node [vertex, label=right:{$v_1$}] (1) at (2,0){};
\node [vertex, label=below:{$v_2$}] (2) at (0,-2){};
\node [vertex, label=left:{$v_3$}] (3) at (-2,0){};
\node [vertex, label=above:{$v_4$}] (4) at (0,2){};

\draw[edge]    (1)  edge  (2);
\draw[edge]    (2)  edge  (3);
\draw[edge]    (3)  edge  (4);
\draw[edge]    (1)  edge  (4);
\draw[edge]    (1)  edge  (3);

\node (L) at (2.5,-2.6){$(k)$};
\end{scope}

\begin{scope}[xshift=9cm, yshift=2.3cm, scale=0.5] %%%crossed
\node [vertex, label=right:{$v_1$}] (1) at (2,0){};
\node [vertex, label=below:{$v_2$}] (2) at (0,-2){};
\node [vertex, label=left:{$v_3$}] (3) at (-2,0){};
\node [vertex, label=above:{$v_4$}] (4) at (0,2){};

\draw[edge]    (4)  edge  (3);
\draw[edge]    (1)  edge  (3);
\draw[edge]    (2)  edge  (1);
\draw[edge]    (2)  edge  (4);
\draw[edge]    (1)  edge  (4);

\node (L) at (2.5,-2.6){$(l)$};
\end{scope}

%%%%%%%%%%%% RENGLON 4

\begin{scope}[xshift=-9cm, yshift=-1.8cm, scale=0.5] %%%P4 Simple
\node [vertex, label=right:{$v_1$}] (1) at (2,0){};
\node [vertex, label=below:{$v_2$}] (2) at (0,-2){};
\node [vertex, label=left:{$v_3$}] (3) at (-2,0){};
\node [vertex, label=above:{$v_4$}] (4) at (0,2){};

\draw[edge]    (1)  edge  (2);
\draw[edge]    (2)  edge  (3);
\draw[edge]    (3)  edge  (4);

\node (L) at (2.5,-2.6){$(m)$};
\end{scope}

\begin{scope}[xshift=-3cm, yshift=-1.8cm, scale=0.5] %%%P4 crossed
\node [vertex, label=right:{$v_1$}] (1) at (2,0){};
\node [vertex, label=below:{$v_2$}] (2) at (0,-2){};
\node [vertex, label=left:{$v_3$}] (3) at (-2,0){};
\node [vertex, label=above:{$v_4$}] (4) at (0,2){};

\draw[edge]    (1)  edge  (3);
\draw[edge]    (4)  edge  (1);
\draw[edge]    (2)  edge  (4);

\node (L) at (2.5,-2.6){$(n)$};
\end{scope}

\begin{scope}[xshift=3cm, yshift=-1.8cm, scale=0.5] %%%P4 Z
\node [vertex, label=right:{$v_1$}] (1) at (2,0){};
\node [vertex, label=below:{$v_2$}] (2) at (0,-2){};
\node [vertex, label=left:{$v_3$}] (3) at (-2,0){};
\node [vertex, label=above:{$v_4$}] (4) at (0,2){};

\draw[edge]    (4)  edge  (3);
\draw[edge]    (4)  edge  (2);
\draw[edge]    (2)  edge  (1);

\node (L) at (2.5,-2.6){$(o)$};
\end{scope}

\begin{scope}[xshift=9cm, yshift=-1.8cm, scale=0.5] %%%P4 Z'
\node [vertex, label=right:{$v_1$}] (1) at (2,0){};
\node [vertex, label=below:{$v_2$}] (2) at (0,-2){};
\node [vertex, label=left:{$v_3$}] (3) at (-2,0){};
\node [vertex, label=above:{$v_4$}] (4) at (0,2){};

\draw[edge]    (4)  edge  (3);
\draw[edge]    (3)  edge  (1);
\draw[edge]    (2)  edge  (1);

\node (L) at (2.5,-2.6){$(p)$};
\end{scope}

%%%%%%%%%%%% RENGLON 5

\begin{scope}[xshift=-6cm, yshift=-5.9cm, scale=0.5] %%%P3 top
\node [vertex, label=right:{$v_1$}] (1) at (2,0){};
\node [vertex, label=below:{$v_2$}] (2) at (0,-2){};
\node [vertex, label=left:{$v_3$}] (3) at (-2,0){};
\node [vertex, label=above:{$v_4$}] (4) at (0,2){};

\draw[edge]    (2)  edge  (3);
\draw[edge]    (3)  edge  (1);

\node (L) at (2.5,-2.6){$(q)$};
\end{scope}

\begin{scope}[xshift=-0cm, yshift=-5.9cm, scale=0.5] %%%P3 middle
\node [vertex, label=right:{$v_1$}] (1) at (2,0){};
\node [vertex, label=below:{$v_2$}] (2) at (0,-2){};
\node [vertex, label=left:{$v_3$}] (3) at (-2,0){};
\node [vertex, label=above:{$v_4$}] (4) at (0,2){};

\draw[edge]    (2)  edge  (3);
\draw[edge]    (3)  edge  (4);

\node (L) at (2.5,-2.6){$(r)$};
\end{scope}

\begin{scope}[xshift=6cm, yshift=-5.9cm, scale=0.5] %%%P3 bottom
\node [vertex, label=right:{$v_1$}] (1) at (2,0){};
\node [vertex, label=below:{$v_2$}] (2) at (0,-2){};
\node [vertex, label=left:{$v_3$}] (3) at (-2,0){};
\node [vertex, label=above:{$v_4$}] (4) at (0,2){};

\draw[edge]    (4)  edge  (3);
\draw[edge]    (3)  edge  (1);

\node (L) at (2.5,-2.6){$(s)$};
\end{scope}

%%%%%%%%%%%% RENGLON 6

\begin{scope}[xshift=-6cm, yshift=-10cm, scale=0.5] %%%trianguilito top
\node [vertex, label=right:{$v_1$}] (1) at (2,0){};
\node [vertex, label=below:{$v_2$}] (2) at (0,-2){};
\node [vertex, label=left:{$v_3$}] (3) at (-2,0){};
\node [vertex, label=above:{$v_4$}] (4) at (0,2){};

\draw[edge]    (2)  edge  (1);
\draw[edge]    (4)  edge  (1);
\draw[edge]    (2)  edge  (4);
\draw[edge]    (3)  edge  (4);

\node (L) at (2.5,-2.6){$(t)$};
\end{scope}

\begin{scope}[xshift=-0cm, yshift=-10cm, scale=0.5] %%%triangulito middle
\node [vertex, label=right:{$v_1$}] (1) at (2,0){};
\node [vertex, label=below:{$v_2$}] (2) at (0,-2){};
\node [vertex, label=left:{$v_3$}] (3) at (-2,0){};
\node [vertex, label=above:{$v_4$}] (4) at (0,2){};

\draw[edge]    (2)  edge  (1);
\draw[edge]    (4)  edge  (1);
\draw[edge]    (2)  edge  (4);
\draw[edge]    (3)  edge  (1);

\node (L) at (2.5,-2.6){$(u)$};
\end{scope}

\begin{scope}[xshift=6cm, yshift=-10cm, scale=0.5] %%%triangulito bottom
\node [vertex, label=right:{$v_1$}] (1) at (2,0){};
\node [vertex, label=below:{$v_2$}] (2) at (0,-2){};
\node [vertex, label=left:{$v_3$}] (3) at (-2,0){};
\node [vertex, label=above:{$v_4$}] (4) at (0,2){};

\draw[edge]    (2)  edge  (1);
\draw[edge]    (4)  edge  (1);
\draw[edge]    (2)  edge  (4);
\draw[edge]    (3)  edge  (2);

\node (L) at (2.5,-2.6){$(v)$};
\end{scope}

\end{tikzpicture}

\caption{All circularly ordered graphs on $4$ vertices. In all cases,
the circular order is the circular closure of $v_1 \le v_2 \le v_3 \le v_4$.}
\label{fig:allcirc}
\end{figure}
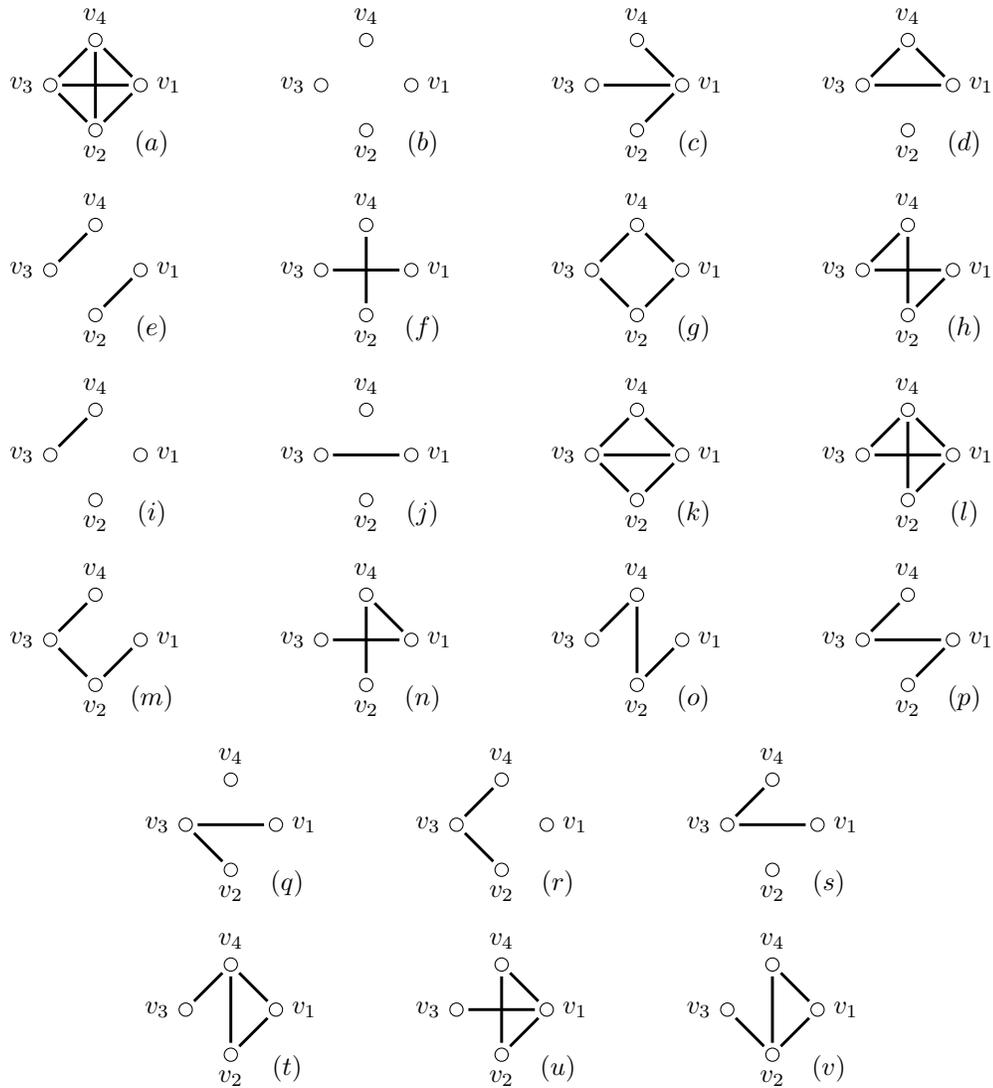
%%%%%%%%%%%%---------%%%%%%%%%------

Note that there are two natural operations on circularly ordered graphs. Let
$(G,c(\le))$ be a circularly ordered graph. The \textit{complement}
$\overline{(G,c(\le))}$ of $(G,c(\le))$ is obtained by taking the complement of
$G$ and respecting the circular order of the vertices,  i.e.,
$\overline{(G,c(\le))} = (\overline{G},c(\le))$. The \textit{reflection} or
\textit{dual} $(G,c(\le))^\ast$ of $(G,c(\le))$ is obtained by considering the
dual $\le^\ast$ of the linear order $\le$ and taking its circular closure, i.e.,
$(G,c(\le))^\ast = (G,c(\le^\ast))$. The latter operation can be interpreted
geometrically as follows. Consider the reflection $r\colon S^1 \to S^1$ over the
$y$-axis and let $(G,c(\le))$ be a circularly ordered graph. If $f \colon V(G)
\to S^1$ is an embedding such that $c(\le)$ is recovered by traversing $S^1$ in
a clockwise motion, then the circular ordering of $V(G)$ in $(G,c(\le))^\ast$ is
recovered by the embedding $r \circ f \colon V(G)\to S^1$ and traversing $S^1$
in a clockwise motion; equivalently, it is recovered by the embedding $f \colon
V(G) \to S^1$ and traversing the circle in an anti-clockwise motion.

For a positive integer $k$, a \textit{simple $k$-path}, $SP_k$, is the $k$-path
$P_k = v_1 \cdots v_k$ together with the circular ordering obtained from the
circular closure of $v_1 \le \cdots \le v_k$. Analogously, if $k \ge 3$ a
\textit{simple $k$-cycle}, $SC_k$, is the $k$-cycle $C_k = v_1 \cdots v_kv_1$
together with the circular ordering obtained from the circular closure of
$v_1\le \cdots \le v_k$. The simple path $SP_4$ and the simple cycle $SC_4$ are
depicted in Figure~\ref{fig:allcirc}, labelled $(m)$ and $(g)$, respectively.
Consider now the five cycle $C_5 = v_1 \cdots v_5 v_1$ and define
\textit{$C_5$-star} as the circularly ordered graph $(C_5,c(v_1 \le v_4 \le v_2
\le v_5 \le v_3))$. Note that the complement of $C_5$-star is $SC_5$, and the
dual of a simple path or a simple cycle is a simple path or a simple cycle,
respectively.

To use a technique analogous to the one used in \cite{feuilloleyJDM} for
depicting families of linearly ordered graphs, we introduce \textit{circularly
ordered patterns}. A \textit{pattern} consists of a set $V$ together with a set
of edges $E$ and a set of non-edges $NE$ with the restriction that $NE \cap E =
\varnothing$. A pattern $(V, E, NE)$ \textit{represents} all graphs $(V(G),
E(G))$ such that $V(G) = V$ and $E \subseteq E(G)$ but $E(G) \cap NE =
\varnothing$. So a circularly ordered pattern $(G,c(\le))$ consists of a pattern
$G$ together with a circular ordering of its vertices, and it represents all
circularly ordered graphs obtained by a graph represented by $G$ and ordering
its vertices by $c(\le)$. Given a set $\mathcal{P}$ of patterns, we say that
$\mathcal{P}$ \textit{generates} all the graphs represented by some pattern in
$\mathcal{P}$. When depicting a pattern we will use straight lines for edges and
dashed lines for no edges. For instance, in Figure~\ref{fig:CG(CA)} we depict a
single circularly ordered pattern and the family of circularly ordered graphs
that it represents. Finally, we say that a circularly ordered graph $(G,c(\le))$
\textit{avoids} a circularly ordered pattern $(H,c(\le'))$ if $(G,c(\le))$
avoids every circularly ordered graph represented by $(H,c(\le'))$.

%%%%----------- Figura con CG(CA)

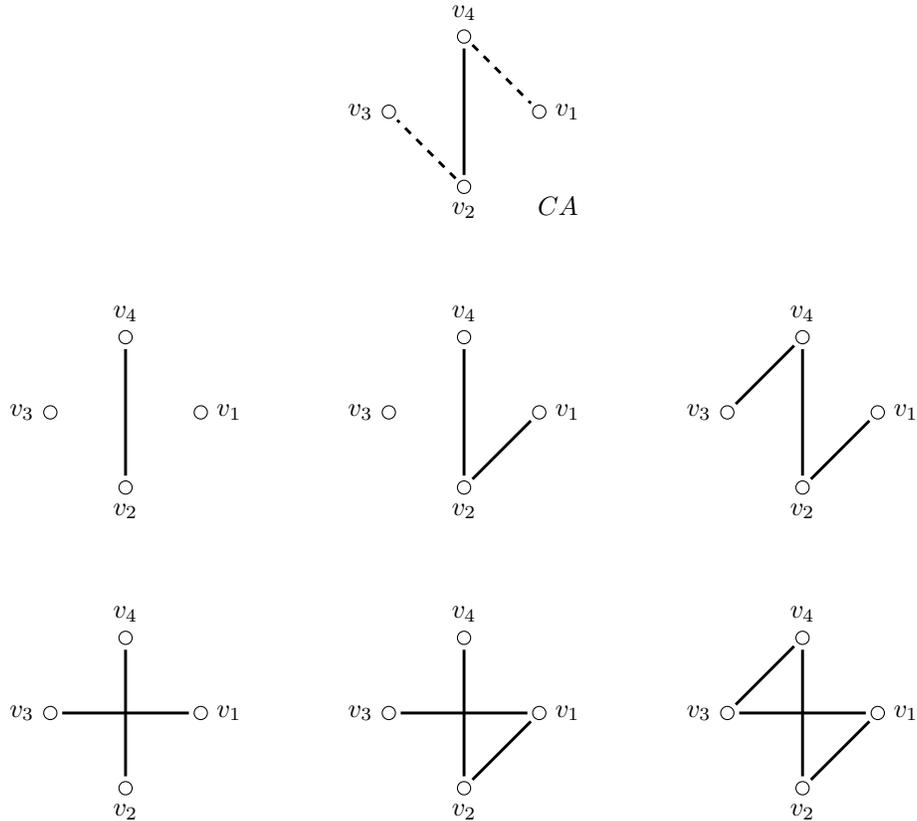
\begin{figure}[ht!]
\begin{center}

\begin{tikzpicture}

\begin{scope}[xshift=0cm, yshift=4cm, scale=0.5]
\node [vertex, label=right:{$v_1$}] (1) at (2,0){};
\node [vertex, label=below:{$v_2$}] (2) at (0,-2){};
\node [vertex, label=left:{$v_3$}] (3) at (-2,0){};
\node [vertex, label=above:{$v_4$}] (4) at (0,2){};

\draw[edge]    (4)  edge  (2);
\draw[edge]    (4)  edge [dashed] (1);
\draw[edge]    (2)  edge [dashed] (3);

\node (L) at (2.5,-2.5){$CA$};
\end{scope}

\begin{scope}[xshift=-4.5cm, yshift=0cm, scale=0.5]
\node [vertex, label=right:{$v_1$}] (1) at (2,0){};
\node [vertex, label=below:{$v_2$}] (2) at (0,-2){};
\node [vertex, label=left:{$v_3$}] (3) at (-2,0){};
\node [vertex, label=above:{$v_4$}] (4) at (0,2){};

\draw[edge]    (4)  edge  (2);
\end{scope}

\begin{scope}[xshift=0cm, yshift=0cm, scale=0.5]
\node [vertex, label=right:{$v_1$}] (1) at (2,0){};
\node [vertex, label=below:{$v_2$}] (2) at (0,-2){};
\node [vertex, label=left:{$v_3$}] (3) at (-2,0){};
\node [vertex, label=above:{$v_4$}] (4) at (0,2){};

\draw[edge]    (4)  edge  (2);
\draw[edge]    (1) edge (2);
\end{scope}

\begin{scope}[xshift=4.5cm, yshift = 0cm, scale=0.5]
\node [vertex, label=right:{$v_1$}] (1) at (2,0){};
\node [vertex, label=below:{$v_2$}] (2) at (0,-2){};
\node [vertex, label=left:{$v_3$}] (3) at (-2,0){};
\node [vertex, label=above:{$v_4$}] (4) at (0,2){};

\draw[edge]    (4)  edge  (2);
\draw[edge]    (1) edge (2);
\draw[edge]    (3) edge (4);
\end{scope}

\begin{scope}[xshift=-4.5cm, yshift=-4cm, scale=0.5]
\node [vertex, label=right:{$v_1$}] (1) at (2,0){};
\node [vertex, label=below:{$v_2$}] (2) at (0,-2){};
\node [vertex, label=left:{$v_3$}] (3) at (-2,0){};
\node [vertex, label=above:{$v_4$}] (4) at (0,2){};

\draw[edge]    (4)  edge  (2);
\draw[edge]    (1)  edge  (3);
\end{scope}

\begin{scope}[xshift=0cm, yshift=-4cm, scale=0.5]
\node [vertex, label=right:{$v_1$}] (1) at (2,0){};
\node [vertex, label=below:{$v_2$}] (2) at (0,-2){};
\node [vertex, label=left:{$v_3$}] (3) at (-2,0){};
\node [vertex, label=above:{$v_4$}] (4) at (0,2){};

\draw[edge]    (4)  edge  (2);
\draw[edge]    (1) edge (2);
\draw[edge]    (1)  edge  (3);

\end{scope}

\begin{scope}[xshift=4.5cm, yshift = -4cm, scale=0.5]
\node [vertex, label=right:{$v_1$}] (1) at (2,0){};
\node [vertex, label=below:{$v_2$}] (2) at (0,-2){};
\node [vertex, label=left:{$v_3$}] (3) at (-2,0){};
\node [vertex, label=above:{$v_4$}] (4) at (0,2){};

\draw[edge]    (4)  edge  (2);
\draw[edge]    (1) edge (2);
\draw[edge]    (3) edge (4);
\draw[edge]    (1)  edge  (3);
\end{scope}

\end{tikzpicture}

\caption{On top, the circularly ordered pattern $CA$. On the two bottom rows,
the family of circularly ordered graphs represented by $CA$.}
\label{fig:CG(CA)}
\end{center}
\end{figure}
%%%%%%%%%%%%---------%%%%%%%%%------

As a consequence of a result due to Tucker \cite{tuckerBAMS76}, we describe a
circularly ordered pattern $CA$ (top of Figure \ref{fig:CG(CA)}), such that the
family of circular-arc graphs is the class of graphs that admit a $CA$-free
circular ordering.

\begin{theorem}\cite{tuckerBAMS76}
\label{thm:tucker}
If $G$ is a graph, then $G$ is a circular-arc graph if and only if the vertices
of $G$ can be arranged in a circular ordering $v_1,\dots,v_n$ such that, for
$i<j$, if $v_iv_j\in E(G)$ then either $v_{i+1},\dots,v_j\in N(v_i)$ or
$v_{j+1},\dots, v_i\in N(v_j)$.
\end{theorem}

Denote by $CP$ the property of circularly ordered graphs described in
Theorem~\ref{thm:tucker}.

\begin{proposition}
\label{prop:CA-CG}
A graph $G$ admits a $CA$-free circular ordering if and only if it is a
circular-arc graph.
\end{proposition}

\begin{proof}
Note that a circular ordering $c(\le)$ of $V(G)$ fails to satisfy $CP$ if and
only if there are four vertices $v_i \le v_k \le v_j \le v_l \le v_i$ such that
$v_iv_j \in E(G)$, and $v_iv_k, v_jv_l \notin E(G)$. Thus, $(G,c(\le))$
satisfies $CP$ if and only if every of its induced circularly ordered subgraphs
on four vertices satisfy $CP$. The statement of this proposition follows since
the family represented by $CA$ corresponds to those circularly ordered graphs on
four  vertices that do not satisfy $CP$.
\end{proof}

We denote by $cr$ the circularly ordered graph labelled $(f)$ in
Figure~\ref{fig:allcirc}, and by $CR$ the set of circularly ordered spanning
supergraphs of $cr$.

\begin{proposition}
\label{prop:outerplanar}
A graph $G$ is an outerplanar graph if and only if it admits a $CR$-free
circular ordering.
\end{proposition}

\begin{proof}
Suppose that a graph $G$ admits a circular ordering $C_G$ of $V(G)$ that avoids
$CR$. Represent the circular ordering $C_G$ by an injective function $f \colon
V(G)\to S^1$. Consider the embedding of $G$ into $\mathbb{R}^2$
obtained from $f$ and representing every edge $xy$ by the segment joining $f(x)$
and $f(y)$. Since $(G,C_G)$ is $CR$-free, then the previously mentioned
embedding has no crossing edges and thus is a planar embedding of $G$. Moreover,
as all edges are represented by a line segment in the interior of $S^1$ and all
vertices are represented by a point on $S^1$, then the embedding is an
outerplanar embedding of $G$. Thus, $G$ is an outerplanar graph.

On the other hand, let $G$ be an outerplanar graph and $G'$ be an outerplanar
embedding of the graph resulting of adding edges to $G$ until it is a
biconnected outerplanar graph. If $C$ is a hamiltonian cycle of $G'$, then a
circular ordering $C_G$ of $V(G)$ is obtained by traversing $C$ in a clockwise
motion. The fact that $(G,C_G)$ is a $CR$-free circular ordering of $G$
follows from the definition of $CR$ and the fact that $G'$ is an outerplanar
embedding of a supergraph of $G$.
\end{proof}

%%%%%%%%%%%%%%%%%%%%%
%%%%-----------              Ordered patterns
%%%%%%%%%%%%%%%%%%%%%

\section{Circular arrangements and linearly ordered patterns}
\label{sec:ordpat}

As noted by Habib and Feuilloley \cite{feuilloleyJDM}, an obvious line of
research in the context of forbidden linearly ordered graphs, is to study
hereditary properties characterized by forbidden sets of linearly ordered graphs
on four vertices or more. To this end, we notice that for any hereditary
property described by a finite set of forbidden circularly ordered graphs, there
is a set of linearly ordered graphs (with the same size of vertex sets) that
describes the same property. Let $c$ be the function that maps a linearly
ordered graph $(G,\le)$ to the circularly ordered graph $(G,c(\le))$, i.e.,
$c(G,\le) = (G,c(\le))$.   The function $c$ can be naturally extended to take
linearly ordered patterns as an argument if we think a linearly ordered pattern
as the set of linearly ordered graphs that it represents. As the following
observation shows, the inverse image of a set of circularly ordered graphs $F$
under $c$, directly relates the families of graphs admitting an $F$-free
circular ordering and those admitting an $(c^{-1}[F])$-circular ordering.  For
this reason, it is convenient to define the ``linearizing operator'' $L$ for a
set of circularly ordered graphs $F$ as $L(F) = c^{-1}[F]$.   Again, $L$ can
take a circularly ordered pattern as an argument if we think it as the set of
circularly ordered graphs it represents.

\begin{observation}
\label{obs:circtolinear}
Let $F$ be a set of circularly ordered graphs and let $\mathcal{P}$ be the class
of graphs that admit an $F$-free circular ordering. Then, $\mathcal{P}$ is the
class of graphs that admit a $L(F)$-free linear ordering.
\end{observation}

\begin{proof}
Recall that every circular ordering can be described as the circular closure of
some linear ordering. So let $(G,c(\le))$ be an $F$-free circular ordering of a
graph $G$. Then, $(G,\le)$ is a $L(F)$-free linear ordering of $G$.
Conversely, if $(G,\le)$ is a $L(F)$-free linear ordering of $G$ then
$(G,c(\le))$ is an $F$-free circular ordering of $G$.
\end{proof}

In particular, since we already showed that outerplanar graphs can be naturally
described by forbidden circularly ordered graphs (Proposition%
~\ref{prop:outerplanar}) by Observation~\ref{obs:circtolinear}, we recover an
observation mentioned in \cite{feuilloleyJDM} that states that there is a finite
set of linearly ordered patterns that characterizes outerplanar graphs. The
class of circular-arc graphs is also described by finitely many forbidden
circularly ordered graphs (Proposition~\ref{prop:CA-CG}) so there is a set of
linearly ordered patterns on four vertices $F_C$ such that the class of graphs
that admit an $F_C$-free linear ordering is the class of circular-arc graphs.
This remark positively answers a question posed by Damaschke: is there a finite
set of linearly ordered graphs that describes the class of circular-arc graphs?
\cite{damaschkeTCGT1990}. To precisely determine $F_C$, let $L(CA)$ be the set
of ordered graphs $(G,\le)$ such that $(G,c(\le))$ is represented by the
circularly ordered pattern $CA$. We depict a pair of linearly ordered patterns
that generate $L(CA)$ in Figure~\ref{fig:LCA}.

%%%%----------- TABLA CON L(CA)

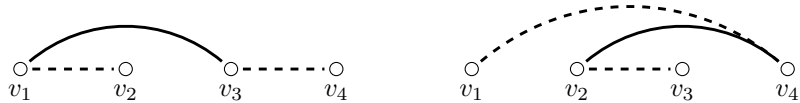
\begin{figure}[ht!]
\begin{center}

\begin{tikzpicture}

\begin{scope}[xshift=-3cm, yshift=0cm, scale=0.7]%P_4
\node [vertex, label=below:{$v_1$}] (1) at (-3,0){};
\node [vertex, label=below:{$v_2$}] (2) at (-1,0){};
\node [vertex, label=below:{$v_3$}] (3) at (1,0){};
\node [vertex, label=below:{$v_4$}] (4) at (3,0){};

\draw[edge]    (1)  edge [bend right=-40]  (3);
\draw[edge]    (1)  edge [dashed] (2);
\draw[edge]    (3)  edge [dashed] (4);
\end{scope}

\begin{scope}[xshift=3cm, yshift = 0cm, scale=0.7]%P2
\node [vertex, label=below:{$v_1$}] (1) at (-3,0){};
\node [vertex, label=below:{$v_2$}] (2) at (-1,0){};
\node [vertex, label=below:{$v_3$}] (3) at (1,0){};
\node [vertex, label=below:{$v_4$}] (4) at (3,0){};

\draw[edge]    (4)  edge [bend right=40]  (2);
\draw[edge]    (2)  edge [dashed]  (3);
\draw[edge]    (1)  edge [bend right=-40, dashed] (4);
\end{scope}

\end{tikzpicture}

\caption{A pair of linearly ordered patterns that generate $L(CA)$.}
\label{fig:LCA}
\end{center}
\end{figure}
%%%%%%%%%%%%---------%%%%%%%%%------

\begin{proposition}
A graph $G$ admits a $L(CA)$-free linear ordering if and only if $G$ is a
circular-arc
graph.
\end{proposition}

Note that in this case, describing the class of circular-arc graphs by forbidden
circular arrangements yields a simpler set of forbidden structures than
describing them by forbidden linearly ordered patterns.

Observation~\ref{obs:circtolinear} gives us the opportunity to propose what we
think is a really interesting question: Is its ``converse'' true? This is, given
a finite set of linearly ordered patterns, $F$, is there a finite set of
circular ordered graphs $F'$ such that the class of graphs that admit an
$F$-free linear ordering is precisely the class of graphs that admit an
$F'$-free circular ordering? We do not have an example where it does not hold,
but the existence of one seems to be likely, so it would be interesting to see
one. In contrast, it is not hard to find examples of some nice classes where
the converse holds, we now present a handful.

Linear forests, caterpillar forests and forests are examples of graph classes
characterized by a set of forbidden linearly ordered patterns on three vertices
\cite{feuilloleyJDM}.

Let $LF$ be the set of circularly ordered graphs that consists of the simple
triangle, both circular orderings of $C_4$, the simple $P_4$, the
\textit{crossed $P_4$}, and the unique circular ordering of the claw. We
illustrate these graphs in Figure~\ref{fig:LF}.

%%%%%%%%%%%%%%%%%%%%%
%%%%%%%%%%%%%%%%%%%%%

\begin{figure}[ht!]
\centering

\begin{tikzpicture}

%%%%%% RENGLON 1

\begin{scope}[xshift=-4.5cm, yshift=2cm, scale=0.5] %%%K_3
\node [vertex, label=right:{$v_1$}] (1) at (2,0){};
\node [vertex, label=below:{$v_2$}] (2) at (0,-2){};
\node [vertex, label=left:{$v_3$}] (3) at (-2,0){};

\draw[edge]    (1)  edge  (2);
\draw[edge]    (1)  edge  (3);
\draw[edge]    (2)  edge  (3);
\end{scope}

\begin{scope}[xshift=0cm, yshift=2cm, scale=0.5] %%%C4 simple
\node [vertex, label=right:{$v_1$}] (1) at (2,0){};
\node [vertex, label=below:{$v_2$}] (2) at (0,-2){};
\node [vertex, label=left:{$v_3$}] (3) at (-2,0){};
\node [vertex, label=above:{$v_4$}] (4) at (0,2){};

\draw[edge]    (1)  edge  (2);
\draw[edge]    (2)  edge  (3);
\draw[edge]    (3)  edge  (4);
\draw[edge]    (1)  edge  (4);
\end{scope}

\begin{scope}[xshift=4.5cm, yshift=2cm, scale=0.5] %%%C4 crossed
\node [vertex, label=right:{$v_1$}] (1) at (2,0){};
\node [vertex, label=below:{$v_2$}] (2) at (0,-2){};
\node [vertex, label=left:{$v_3$}] (3) at (-2,0){};
\node [vertex, label=above:{$v_4$}] (4) at (0,2){};

\draw[edge]    (4)  edge  (3);
\draw[edge]    (1)  edge  (3);
\draw[edge]    (2)  edge  (1);
\draw[edge]    (2)  edge  (4);
\end{scope}

%%%%%%%%%%%% RENGLON 2

\begin{scope}[xshift=-4.5cm, yshift=-2cm, scale=0.5]%P_4 C
\node [vertex, label=right:{$v_1$}] (1) at (2,0){};
\node [vertex, label=below:{$v_2$}] (2) at (0,-2){};
\node [vertex, label=left:{$v_3$}] (3) at (-2,0){};
\node [vertex, label=above:{$v_4$}] (4) at (0,2){};

\draw[edge]    (4)  edge  (3);
\draw[edge]    (3)  edge  (2);
\draw[edge]    (2)  edge  (1);
\end{scope}

\begin{scope}[xshift=0cm, yshift=-2cm, scale=0.5] %%%P4 crossed
\node [vertex, label=right:{$v_1$}] (1) at (2,0){};
\node [vertex, label=below:{$v_2$}] (2) at (0,-2){};
\node [vertex, label=left:{$v_3$}] (3) at (-2,0){};
\node [vertex, label=above:{$v_4$}] (4) at (0,2){};

\draw[edge]    (1)  edge  (3);
\draw[edge]    (4)  edge  (1);
\draw[edge]    (2)  edge  (4);
\end{scope}

\begin{scope}[xshift=4.5cm, yshift=-2cm, scale=0.5] %%%Claw
\node [vertex, label=right:{$v_1$}] (1) at (2,0){};
\node [vertex, label=below:{$v_2$}] (2) at (0,-2){};
\node [vertex, label=left:{$v_3$}] (3) at (-2,0){};
\node [vertex, label=above:{$v_4$}] (4) at (0,2){};

\draw[edge]    (3)  edge  (2);
\draw[edge]    (1)  edge  (3);
\draw[edge]    (3)  edge  (4);
\end{scope}
\end{tikzpicture}

\caption{All circularly ordered graphs in $LF$. If we do not use the last graph
(the claw), the resulting family is $CF$.}
\label{fig:LF}
\end{figure}
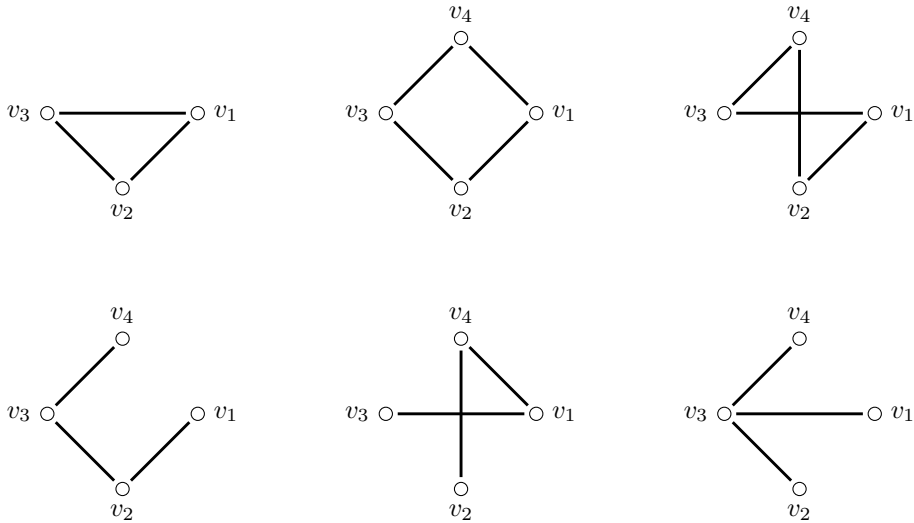
%%%%%%%%%%%%---------%%%%%%%%%------

It is convenient to define the following class of circularly ordered paths.
First note that there are four possible non-isomorphic circular orderings of
$P_4$, namely $SP_4$ (Figure \ref{fig:allcirc}.$m$), the crossed $P_4$ (Figure
\ref{fig:allcirc}.$n$) and two more which we will denote by $Z$ (Figure
\ref{fig:allcirc}.$o$) and $Z^\ast$ (Figure \ref{fig:allcirc}.$p$). Note that
the dual of $Z$ is $Z^\ast$ (which justifies our choice of notation). Given a
positive integer $k$ greater than $3$, a \textit{$k$-zigzag} is a circular
ordering of the $k$-path $P$ such that every induced copy of $P_4$ in $P$ is
ordered as $Z$ or $Z^\ast$. In particular, $Z$ and $Z^\ast$, are the unique
$4$-zigzags. Finally, we say that a circularly ordered graph $G' = (G,C)$ has a
pair of \textit{crossing edges} if there is a pair of edges $v_1v_3$ and
$v_2v_4$ of $G$, such that $(v_1,v_2,v_3) \in C$ and $(v_3,v_4,v_1) \in C$;
otherwise, we say that $G'$ has  no crossing edges. For instance, the crossed
$P_4$ is a circular ordering of $P_4$ with crossing edges, while the other three
circular orderings of $P_4$ have no crossing edges.

\begin{observation}
\label{obs:zigzag}
Let $k$ be a positive integer, $k\ge 4$. A circular arrangement $P_k'$ of $P_k$
is $LF$-free if and only if it is a $k$-zigzag.
\end{observation}

\begin{proof}
Since there are exactly four circular orderings of $P_4$, and two of them,
namely the simple $P_4$ and the crossed $P_4$ are members of $LF$, then the
desired result follows directly from the definition of a $k$-zigzag.
\end{proof}

This simple observation yields the following statement.

\begin{proposition}
\label{prop:LF}
A graph $G$ is a linear forest if and only if it admits an $LF$-free circular
ordering.
\end{proposition}

\begin{proof}
One implication follows from the fact that for every positive integer $k$, a
$k$-zigzag is an $LF$-free circular ordering of $P_k$. To prove the converse
implication first note that if a graph $G$ has a vertex of degree at least $3$,
then $G$ contains either a claw or a triangle. Since the unique circular
ordering of both of these graphs belongs to $LF$, if a graph $G$ admits an
$LF$-free circular ordering then $G$ is a disjoint union of paths and cycles
with no triangles. Again, as both circular orderings of $C_4$ belong to $LF$,
any graph that admits an $LF$-free ordering is $C_4$-free. Now we show that for
every positive integer $k$, $k\ge 5$, the $k$-cycle, $C_k = v_1\cdots v_kv_1$,
does not admit an $LF$-free circular ordering. By Observation~\ref{obs:zigzag},
if $C_k$ admits an $LF$-free circular ordering, $C_k'$, then the induced path
$v_1 \cdots v_{k-1}$ must be arranged as a $k$-zigzag. But then, wherever $v_k$
is placed in the circular ordering it forces $C_k'$ to have an induced copy of
$P_4$ with crossing edges, contradicting the fact that $C_k'$ is an $LF$-free
ordering.
\end{proof}

Let $CF$ be the set obtained from $LF$ by removing the claw (see Figure
\ref{fig:LF}) and let $T_2$ be the graph obtained from the claw by subdividing
every edge. Recall that a graph $G$ is a caterpillar if and only if it is a
$T_2$-free tree.

\begin{proposition}
A graph $G$ is a caterpillar forest if and only if it admits a $CF$-free
circular ordering.
\end{proposition}

\begin{proof}
First note that every caterpillar forest is an induced subgraph of a
caterpillar, thus it suffices to observe that every caterpillar admits a
$CF$-free circular ordering. We order the largest dominating path, $ P = v_1
\cdots v_k$, as a $k$-zigzag. Note that for every $j \in \{2,\cdots,k-1\}$ one
of the circular segments delimited by $v_{j-1}$ and $v_{j+1}$ contains no
vertices of $P$. We place the leaves adjacent to $v_j$ in this circular segment.
It is not hard to observe that this circular ordering of a caterpillar if
$CF$-free. On the contrary if $G$ is not a caterpillar forest then it must
contain a cycle or a $T_2$. With the same arguments as in the proof of
Proposition~\ref{prop:LF} one can notice that no cycle admits a $CF$-free
circular ordering. It is also not hard to observe that $T_2$ does not admit a
$CF$-free circular ordering, which concludes the proof.
\end{proof}

Now we show that forests can be characterized by a finite set of forbidden
circularly ordered graphs. Let $\mathcal{F}$ be the set of all circular patterns
depicted in Figure~\ref{fig:F}. In particular, every $\mathcal{F}$-free circular
ordered graph must avoid crossing edges. Thus, if we were to obtain an
$\mathcal{F}$-free circular ordering of a cycle, we should order its vertices
cyclicly, but then we would obtain either one of the cycles in $\mathcal{F}$ or
the simple $P_5$. Hence, no cycle admits an $\mathcal{F}$-free circular
ordering. We will show that every forest does admit an $\mathcal{F}$-free
circular ordering.

%%%%%%%%%%%%%%%%%%%%%
%%%%%%%%%%%%%%%%%%%%%

\begin{figure}[ht!]
\centering
\begin{tikzpicture}
\begin{scope}[scale=0.75]

%%%%%% RENGLON 1

\begin{scope}[xshift=-6cm, yshift=2cm, scale=0.5] %%%K_3
\node [vertex, label=right:{$v_1$}] (1) at (2,0){};
\node [vertex, label=below:{$v_2$}] (2) at (0,-2){};
\node [vertex, label=left:{$v_3$}] (3) at (-2,0){};

\draw[edge]    (1)  edge  (2);
\draw[edge]    (1)  edge  (3);
\draw[edge]    (2)  edge  (3);
\end{scope}

\begin{scope}[xshift=-2cm, yshift=2cm, scale=0.5] %%%C4 simple
\node [vertex, label=right:{$v_1$}] (1) at (2,0){};
\node [vertex, label=below:{$v_2$}] (2) at (0,-2){};
\node [vertex, label=left:{$v_3$}] (3) at (-2,0){};
\node [vertex, label=above:{$v_4$}] (4) at (0,2){};

\draw[edge]    (1)  edge  (2);
\draw[edge]    (2)  edge  (3);
\draw[edge]    (3)  edge  (4);
\draw[edge]    (1)  edge  (4);
\end{scope}

\begin{scope}[xshift=2cm, yshift=2cm, scale=0.5] %%%P4 Crossed
\node [vertex, label=right:{$v_1$}] (1) at (2,0){};
\node [vertex, label=below:{$v_2$}] (2) at (0,-2){};
\node [vertex, label=left:{$v_3$}] (3) at (-2,0){};
\node [vertex, label=above:{$v_4$}] (4) at (0,2){};

\draw[edge]    (4)  edge  (3);
\draw[edge]    (1)  edge  (3);
\draw[edge]    (2)  edge  (1);
\draw[edge]    (2)  edge  (4);
\end{scope}

\begin{scope}[xshift=6cm, yshift=2cm, scale=0.5]%Tache
\node [vertex, label=right:{$v_1$}] (1) at (2,0){};
\node [vertex, label=below:{$v_2$}] (2) at (0,-2){};
\node [vertex, label=left:{$v_3$}] (3) at (-2,0){};
\node [vertex, label=above:{$v_4$}] (4) at (0,2){};

\draw[edge]    (1)  edge  (3);
\draw[edge]    (4)  edge  (2);
\end{scope}

%%%%%%%%%%%% RENGLON 2

\begin{scope}[xshift=-4.5cm, yshift=-2cm, scale=0.5] %%%P4 crossed
\node [vertex, label=right:{$v_1$}] (1) at (2,0){};
\node [vertex, label=below:{$v_2$}] (2) at (0,-2){};
\node [vertex, label=left:{$v_3$}] (3) at (-2,0){};
\node [vertex, label=above:{$v_4$}] (4) at (0,2){};

\draw[edge]    (1)  edge  (3);
\draw[edge]    (4)  edge  (1);
\draw[edge]    (2)  edge  (4);
\end{scope}

\begin{scope}[xshift=0cm, yshift=-2cm, scale=0.5] %%%C5 simple
\node [vertex, label=right:{$v_1$}] (1) at (0:3){};
\node [vertex, label=above:{$v_2$}] (2) at (72:3){};
\node [vertex, label=left:{$v_3$}] (3) at (144:3){};
\node [vertex, label=left:{$v_4$}] (4) at (216:3){};
\node [vertex, label=below:{$v_5$}] (5) at (288:3){};

\draw[edge]    (1)  edge  (2);
\draw[edge]    (2)  edge  (3);
\draw[edge]    (3)  edge  (4);
\draw[edge]    (1)  edge  (5);
\draw[edge]    (4)  edge  (5);
\end{scope}

\begin{scope}[xshift=4.5cm, yshift=-2cm, scale=0.5] %%%P5 simple
\node [vertex, label=right:{$v_1$}] (1) at (0:3){};
\node [vertex, label=above:{$v_2$}] (2) at (72:3){};
\node [vertex, label=left:{$v_3$}] (3) at (144:3){};
\node [vertex, label=left:{$v_4$}] (4) at (216:3){};
\node [vertex, label=below:{$v_5$}] (5) at (288:3){};

\draw[edge]    (1)  edge  (2);
\draw[edge]    (2)  edge  (3);
\draw[edge]    (3)  edge  (4);
\draw[edge]    (4)  edge  (5);
\end{scope}

\end{scope}
\end{tikzpicture}

\caption{All circularly ordered graphs in $F$.}
\label{fig:F}
\end{figure}
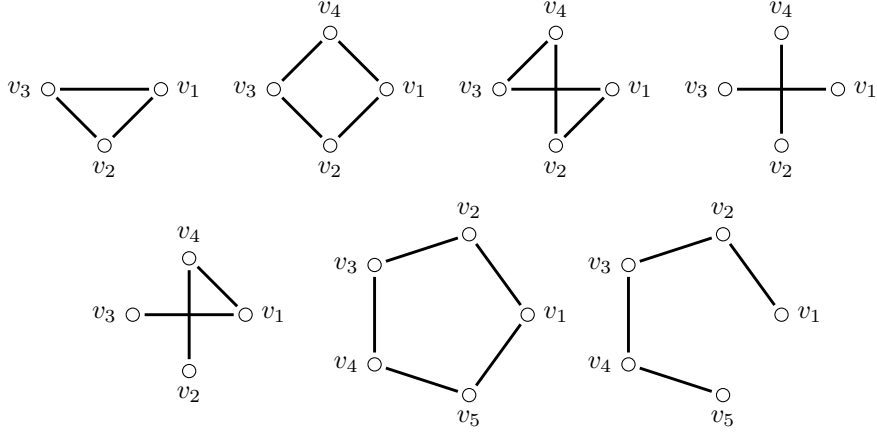
%%%%%%%%%%%%---------%%%%%%%%%------

\begin{theorem}
\label{thm:trees}
A graph $G$ admits an $\mathcal{F}$-free circular ordering if and only if $G$ is
a forest.
\end{theorem}

\begin{proof}
We have already shown that if $G$ admits an $\mathcal{F}$-free circular
ordering, then $G$ is an acyclic graph. To observe that every forest admits such
a circular ordering it suffices to note that every tree does. Indeed, every
forest is an induced subgraph of some tree and the class of graphs that admit an
$\mathcal{F}$-free circular ordering is a hereditary class of graphs.

Given a pair of vertices, $x$ and $y$, whenever we say we place $y$ ``ahead of''
(``behind'') $x$, we think we are traversing the circle in a clockwise motion
starting from $x$ and we place $y$ before seeing any other vertex (after seeing
all other vertices).

Let $T$ be a tree. We will describe the circular ordering of $V(T)$ by arranging
the vertices of $T$ around the circle and we will construct this arrangement
recursively. Let $\{v_0,\dots, v_{n-1}\}$ be an enumeration of the vertices of
$T$ in such a way that if $i \le j$ then $d(v_i,v_0) \le d(v_j,v_0)$. In
particular, the graph $T_k$ induced by $\{v_0, \dots, v_k\}$ is a tree for every
$k \in \{0,\dots, n-1\}$ where $v_k$ is a leaf of $T_k$. We first place the root
$v_0$ anywhere in the circle. Suppose we have arranged $V(T_{k-1})$, now we
arrange $V(T_k)$ by respecting the ordering of $V(T_{k-1})$ and simply including
$v_k$ as follows. Let $a$ be the ancestor of $v_k$. If $a = v_0$ then
incorporate $v_k$ behind $v_0$. On the other hand, let $b$ be the ancestor of
$a$. If $b = v_0$ then include $v_k$ behind $a$. Finally, if $b \ne v_0$ let $c$
be the ancestor of $b$. There are two cases:

\begin{itemize}
	\item when traversing the circle in a clockwise motion we see $(a,b,c)$, in
    this case we include $v_k$ ahead of $a$, or

	\item when traversing the circle in a clockwise motion we see $(a,c,b)$, in
    this case we include $v_k$ behind $a$.
\end{itemize}

We illustrate this construction in Figure~\ref{fig:recConst}.

%%%%----------- Figura construccion recursiva

\begin{figure}[ht!]
\centering

\begin{tikzpicture}

\begin{scope}[xshift=-4.5cm, yshift=0cm, scale=0.5]
\node [vertex, fill, label=90:{$v_0$}] (0) at (90:3){};
\end{scope}

\begin{scope}[xshift=0cm, yshift=0cm, scale=0.5]
\node [vertex, label=90:{$a = v_0$}] (0) at (90:3){};
\node [vertex] (1) at (306:3){};
\node [vertex] (2) at (198:3){};
\node [vertex] (3) at (54:3){};
\node [vertex] (4) at (270:3){};
\node [vertex, fill, label=126:{$v_k$}] (5) at (126:3){};
\node [vertex] (6) at (346:3){};
\node [vertex] (7) at (234:3){};
\node [vertex] (8) at (18:3){};
\node [vertex] (9) at (162:3){};

\draw[edge]    (0)  edge  (5);
\end{scope}

\begin{scope}[xshift=4.5cm, yshift = 0cm, scale=0.5]
\node [vertex, label=90:{$b = v_0$}] (0) at (90:3){};
\node [vertex] (1) at (306:3){};
\node [vertex] (2) at (198:3){};
\node [vertex] (3) at (54:3){};
\node [vertex, fill, label=270:{$v_k$}] (4) at (270:3){};
\node [vertex] (5) at (126:3){};
\node [vertex] (6) at (346:3){};
\node [vertex, label=234:{$a$}] (7) at (234:3){};
\node [vertex] (8) at (18:3){};
\node [vertex] (9) at (162:3){};

\draw[edge]    (0)  edge  (7);
\draw[edge]    (4)  edge  (7);
\end{scope}

%%%%%Renglon 2

\begin{scope}[xshift=-3cm, yshift=-5cm, scale=0.5]
\node [vertex] (0) at (90:3){};
\node [vertex, label=306:{$c$}] (1) at (306:3){};
\node [vertex, label=198:{$a$}] (2) at (198:3){};
\node [vertex, label=54:{$b$}] (3) at (54:3){};
\node [vertex] (4) at (270:3){};
\node [vertex] (5) at (126:3){};
\node [vertex] (6) at (346:3){};
\node [vertex] (7) at (234:3){};
\node [vertex] (8) at (18:3){};
\node [vertex, fill, label=162:{$v_k$}] (9) at (162:3){};

\draw[edge]    (2)  edge  (3);
\draw[edge]    (9)  edge  (2);
\draw[edge]    (3)  edge  (1);
\end{scope}

\begin{scope}[xshift=3cm, yshift=-5cm, scale=0.5]
\node [vertex] (0) at (90:3){};
\node [vertex] (1) at (306:3){};
\node [vertex, label=198:{$a$}] (2) at (198:3){};
\node [vertex, label=54:{$b$}] (3) at (54:3){};
\node [vertex] (4) at (270:3){};
\node [vertex, label=126:{$c$}] (5) at (126:3){};
\node [vertex] (6) at (346:3){};
\node [vertex, fill, label=234:{$v_k$}] (7) at (234:3){};
\node [vertex] (8) at (18:3){};
\node [vertex] (9) at (162:3){};

\draw[edge]    (2)  edge  (7);
\draw[edge]    (3)  edge  (2);
\draw[edge]    (3)  edge  (5);
\end{scope}

\end{tikzpicture}

\caption{Five possible steps in the proposed recursive circular arrangement of
a tree.}
\label{fig:recConst}
\end{figure}
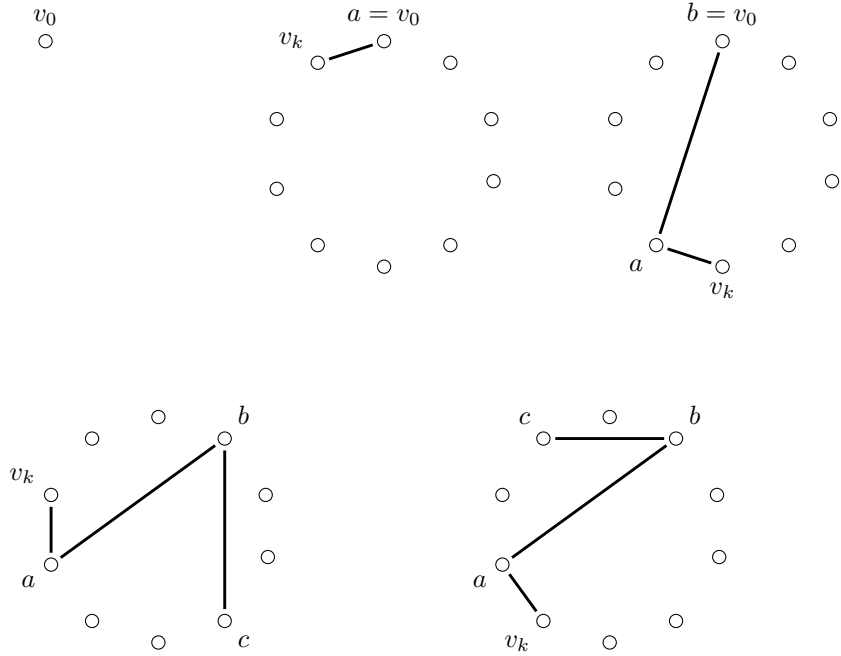
%%%%%%%%%%%%---------%%%%%%%%%------

Let $T_c$ be the tree $T$ together with the previously constructed circular
ordering of $V(T)$. Now we prove that $T_c$ is $\mathcal{F}$-free. Since $T$ is
a tree, $T_c$ avoids every cycle in $\mathcal{F}$. Also, at every step of the
recursive construction, we place the new vertex either ahead or behind its
parent, so there are no crossing edges in $T_c$. Thus, it only remains to verify
that $T_c$ contains no simple $P_5$.

First assume that $T_c$ contains a simple $P_4$, $P_4 = u_1u_2u_3u_4$. Let $i
\in \{1,2,3,4\}$ be the vertex of $P_4$ closest to $v_0$ in $T$ (this index is
unique because otherwise there would be a cycle in $T$). Due to the recursive
rule, it is not hard to notice that $i\in\{2,3\}$. Thus, if $T_c$ contains a
simple $P_5$, say $P_5 = w_1w_2w_3w_4w_5$, then $w_3$ must be the closest vertex
in $P_5$ to $v_0$. So when we added $w_1$ to $T_c$, we included it behind its
parent, $w_2$. Which means that $w_3 = v_0$ or when traversing the circle in a
clockwise motion we see $(w_2,c,w_3)$ where $c$ is the parent of $w_3$. In both
cases, when $w_5$ was included in the arrangement, it was added behind $w_4$.
Then, it means that when we included $w_5$ in the recursion process $w_3$ was
not included yet, but this contradicts the fact the $w_3$ is closer to $v_0$ and
the choice of the order in which we process the vertices of $T$. Therefore,
$T_c$ is $\mathcal{F}$-free.
\end{proof}

Note that the recursive construction of the circular ordering exposed in the
proof of Theorem~\ref{thm:trees}, yields an algorithm to construct an
$\mathcal{F}$-free circular ordering of a tree. This algorithm runs in
polynomial time as we process every vertex only once, and every time we process
a vertex we make a constant amount of operations.

The descriptions by forbidden circular arrangements of outerplanar graphs and
circular-arc graphs proposed in this section are simpler (and somewhat more
intuitive) than their descriptions by forbidden linearly ordered graphs. On the
contrary, describing forests, linear forests and caterpillar forest by linearly
ordered graphs yield simpler expressions (and proofs) than describing these
classes by forbidden circularly ordered graphs. But this should be expected
since these classes are characterized by forbidden linear patterns on three
vertices.  Every graph on three vertices has a unique circular ordering, thus,
forbidding induced circularly ordered graphs on three vertices is equivalent to
forbidding induced graphs on three vertices (without orderings), but none of
these families can be characterized by forbidding induced subgraphs on three
vertices. Nonetheless the statements of this section show that circularly
ordered graphs can describe several natural graph classes. Moreover, these
observations raise the question of whether for any finite set of linearly
ordered patterns $F$ there is a finite set $F'$ of (possible larger) circular
arrangements such that $F$ and $F'$ describe the same classes by forbidden
linearly ordered patterns and forbidden circular arrangements, respectively.

%%%%%%%%%%%%%%%%%%%%%%%%%%%%%
%%%%%%%%%%%%%%%%%%%%%%%%%%%%%
%%%%%%%%%%%%%%%%%%%%%%%%%%%%%

\section{Circular chromatic number and circular orderings}\label{sec:circhrom}

In this section we study how certain forbidden circular orderings relate to the
circular chromatic number of graphs. These forbidden orderings stem from the
following characterization of $k$-colourable graphs in terms of forbidden linear
orderings.

\begin{proposition}\cite{feuilloleyJDM, hellESA2014}
\label{prop:kcollinear}
Let $k$ be a positive integer. A graph $G$ is $k$-colourable if and only if
there is a linear ordering $\le$ of $V(G)$ such that there are no $k+1$ vertices
$v_1 \le \cdots \le v_{k+1}$ such that $v_iv_{i+1}\in E(G)$ for every
$i\in\{1,\dots, k\}$.
\end{proposition}

This result can be restated in terms of homomorphisms. Recall that for graphs
$G$ and $H$, we denote the existence of a homomorphism from $G$ to $H$ by $G \to
H$; we also denote by $G \not \to H$ the fact that there is no homomorphism from
$G$ to $H$.

For a positive integer $k$ denote by $St_k$ the straight path on $k$ vertices,
i.e., $St_k$ has vertex set $\{v_1, \dots, v_k\}$ with the natural ordering of
their indices and with edge set $\{v_1v_2, v_2v_3, \dots, v_{k-1}v_k\}$. Now
Proposition~\ref{prop:kcollinear} can be restated as follows: \textit{a graph
$G$ is $k$-colourable if and only if there is a linear ordering $\le$ of $V(G)$
such that there is no homomorphism (of linearly order graphs) from $St_{k+1}$ to
$(G,\le)$.}

We are interested in proving an analogous version of this result for circular
orderings. Instead of the straight path $St_k$ we consider the simple path
$SP_k$.   (The definition of $SP_k$ is given in Section \ref{sec:Basic},
and the simple path $SP_4$ is depicted in Figure \ref{fig:allcirc} ($m$).) For a
positive integer $k$, $k\ge 2$, we denote by $\mathcal{C}_k$ the class of graphs
$G$ that admit a circular ordering $c(\le)$ such that $SP_k\not\to (G,c(\le))$
(as circularly ordered graphs). We proceed to characterize these classes in
terms of the circular chromatic number, and we begin with the following
observation.

\begin{observation}
\label{obs:hompreimage}
For any positive integer $k$, $k\ge 2$, the class $\mathcal{C}_k$ is closed
under homomorphic pre-images. That is, if a graph $G$ belongs to
$\mathcal{C}_k$, then for any graph $H$ such that $H\to G$ we have that $H\in
\mathcal{C}_k$.
\end{observation}

\begin{proof}
Let $\varphi \colon H \to G$ be a homomorphism. It suffices to order vertices if
$H$ in any way such that for every $x\in V(G)$ the vertices of $H$ in
$\varphi^{-1}(x)$ are contiguous in the circular ordering.
\end{proof}

It is not hard to observe that for every positive integer $k$, $k\ge 2$, there
is a finite set $\mathcal{H}_k$ such that a graph belongs to $\mathcal{C}_k$ if
and only if it admits an $\mathcal{H}_k$-free circular ordering. Indeed,
$\mathcal{H}_k$ can be constructed by first considering the family of all
circularly ordered graphs that are homomorphic images of $SP_k$, then obtaining
$\mathcal{H}_k$ as the antichain of minimal circularly ordered graphs (with
respect to the order of induced circularly ordered graphs) in this family. For
instance, $\mathcal{H}_4$ consists of the triangle, the simple $C_4$ and
$SP_4$. These circularly ordered graphs are depicted in Figure~\ref{fig:H4}.

%%%%----------- Figura con H4

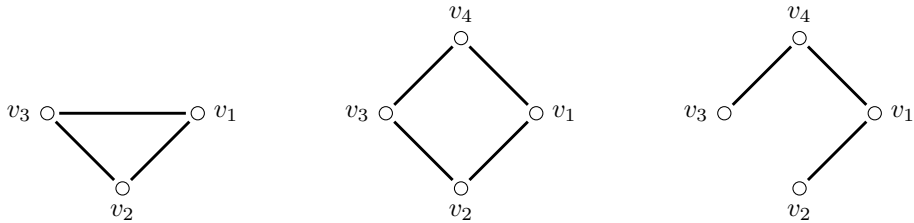
\begin{figure}[ht!]
\centering

\begin{tikzpicture}

\begin{scope}[xshift=-4.5cm, yshift=2cm, scale=0.5] %%%K_3
\node [vertex, label=right:{$v_1$}] (1) at (2,0){};
\node [vertex, label=below:{$v_2$}] (2) at (0,-2){};
\node [vertex, label=left:{$v_3$}] (3) at (-2,0){};

\draw[edge]    (1)  edge  (2);
\draw[edge]    (1)  edge  (3);
\draw[edge]    (2)  edge  (3);
\end{scope}

\begin{scope}[xshift=0cm, yshift=2cm, scale=0.5] %%%C4 simple
\node [vertex, label=right:{$v_1$}] (1) at (2,0){};
\node [vertex, label=below:{$v_2$}] (2) at (0,-2){};
\node [vertex, label=left:{$v_3$}] (3) at (-2,0){};
\node [vertex, label=above:{$v_4$}] (4) at (0,2){};

\draw[edge]    (1)  edge  (2);
\draw[edge]    (2)  edge  (3);
\draw[edge]    (3)  edge  (4);
\draw[edge]    (1)  edge  (4);
\end{scope}

\begin{scope}[xshift=4.5cm, yshift=2cm, scale=0.5] %%%P4 Simple
\node [vertex, label=right:{$v_1$}] (1) at (2,0){};
\node [vertex, label=below:{$v_2$}] (2) at (0,-2){};
\node [vertex, label=left:{$v_3$}] (3) at (-2,0){};
\node [vertex, label=above:{$v_4$}] (4) at (0,2){};

\draw[edge]    (4)  edge  (3);
\draw[edge]    (4)  edge  (1);
\draw[edge]    (2)  edge  (1);
\end{scope}
\end{tikzpicture}

\caption{An illustration of the circularly ordered graphs in $\mathcal{H}_4$.}
\label{fig:H4}
\end{figure}
%%%%%%%%%%%%---------%%%%%%%%%------

We proceed to show that for any positive integer $k$, $k\ge 2$, a graph $G$ with
$\chi_c(G) < k$ must satisfy $G \in \mathcal{C}_{k+1}$. In fact, we can
immediately show that when $k=2$, this condition is not only sufficient, but
also necessary.

\begin{observation}
\label{obs:cric2}
A graph $G$ belongs to $\mathcal{C}_3$ if and only if $\chi_c(G) < 2$.
\end{observation}

\begin{proof}
On one hand, $SP_3$ maps homomorphically to the unique circular ordering of
$K_2$. So $G\in \mathcal{C}_3$ if and only if $G$ has no edges. On the other
hand, if $\sfrac{p}{q} < 2$ then $K_{\sfrac{p}{q}}$ is an edgeless graph. So by
Proposition~\ref{prop:min}, $\chi_c(G) < 2$ if and only if $G$ had no edges.
\end{proof}

Now, for every positive integer $k$, $k\ge 3$, we construct a sequence of graphs
$\{H^k_n\}_{n \ge 1}$ such that $\chi_c(H^k_n) < k$ and $H^k_n\in
\mathcal{C}_{k+1}$ for every $n \ge 1$. Consider a pair of positive integers $n$
and $k$, $k\ge 3$, the graph $H_n^k$ is defined as the rational complete graph
$K_{\sfrac{(kn-1)}{n}}$. In particular, $H^k_1 \cong K_{k-1}$ and $H^k_2 \cong
\overline{C}_{2k-1}$. In Figure~\ref{fig:H2H3} we depict $H^3_2$ and $H^3_3$.
Note that $H^3_3 \cong M_8$, where $M_8$ is the M\"obius ladder on eight
vertices.

%%%%----------- Figura con circulant

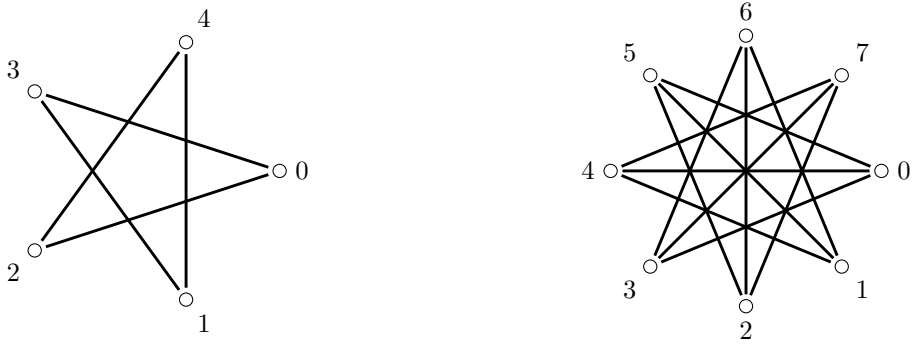
\begin{figure}[ht!]
\centering
\begin{tikzpicture}

\begin{scope}[xshift=-4cm, yshift = -4cm, scale=0.6]
\node [vertex, label=0:{$0$}] (0) at (0:3){};
\node [vertex, label=288:{$1$}] (1) at (288:3){};
\node [vertex, label=216:{$2$}] (2) at (216:3){};
\node [vertex, label=144:{$3$}] (3) at (144:3){};
\node [vertex, label=72:{$4$}] (4) at (72:3){};

\foreach \from/\to in {0/2, 0/3, 1/3, 1/4, 2/4}
  \draw [edge] (\from) to (\to);
\end{scope}

\begin{scope}[xshift=4cm, yshift = -4cm, scale=0.6]
\node [vertex, label=0:{$0$}] (1) at (0:3){};
\node [vertex, label=45:{$7$}] (2) at (45:3){};
\node [vertex, label=90:{$6$}] (3) at (90:3){};
\node [vertex, label=135:{$5$}] (4) at (135:3){};
\node [vertex, label=180:{$4$}] (5) at (180:3){};
\node [vertex, label=225:{$3$}] (6) at (225:3){};
\node [vertex, label=270:{$2$}] (7) at (270:3){};
\node [vertex, label=315:{$1$}] (8) at (315:3){};

\foreach \from/\to in {1/4,1/5,1/6,2/5,2/6,2/7,3/6,3/7,3/8,4/7,4/8,5/8}
  \draw [edge] (\from) to (\to);
\end{scope}

\end{tikzpicture}

\caption{The graphs $H^3_2$ (left) and $H^3_3$ (right).}
\label{fig:H2H3}
\end{figure}
%%%%%%%%%%%%---------%%%%%%%%%------

\begin{lemma}
\label{lem:sequence}
Let $G$ be a graph and $k$ a positive integer, $k\ge 3$. If $\chi_c(G) < k$ then
there is a positive integer $m$ such that $G \to H^k_m$.
\end{lemma}

\begin{proof}
It is not hard to observe that $\{\frac{kn-1}{n}\}_{n \in \mathbb{Z}^+}$ is an
increasing sequence that converges to $k$. So for any rational number $r$ such
that $r < k$, there is a positive integer $m$ such that $r \le \frac{km-1}{m}$.
Consider a graph $G$ such that $\chi_c(G) < k$. By Proposition~\ref{prop:min}
there is a rational number $\sfrac{p}{q}$ such that $G \to K_{\sfrac{p}{q}}$ and
$\sfrac{p}{q} < k$. Let $m$ be a positive integer such that $\sfrac{p}{q} \le
\sfrac{(km-1)}{m} < k$. By Theorem~\ref{thm:pqhoms}, $K_{\sfrac{p}{q}}\to
K_{\sfrac{(km-1)}{m}}$ which concludes the proof since $G \to K_{\sfrac{p}{q}}$
and $H_m^k = K_{\sfrac{(km-1)}{m}}$.
\end{proof}

\begin{proposition}
\label{prop:circhom}
Let $G$ be a graph and $k$ a positive integer, $k\ge 3$. If $\chi_c(G) < k$ then
$G\in \mathcal{C}_{k+1}$.
\end{proposition}

\begin{proof}
By Lemma~\ref{lem:sequence}, if the circular chromatic number of a graph $G$ is
strictly less than $k$, then there is a positive integer $m$ such that $G \to
H_m^k$. Also, recall that, by Observation~\ref{obs:hompreimage}, for every
positive integer $k$ the class of graphs $\mathcal{C}_{k+1}$ is closed under
homomorphic preimages. Hence, if $H_m^k\in \mathcal{C}_{k+1}$ for every $m \ge
1$, then any graph with circular chromatic number strictly less than $k$ belongs
to $\mathcal{C}_k$. So we proceed to prove that for every positive integer $m$
the graph $H_m^k$ belongs to $\mathcal{C}_{k+1}$. To do so, we consider the
canonical circular ordering $c(\le)$ of the vertices of $H_m^k$, i.e., the
circular closure of $0 \le 1 \le \cdots \le km-2$. We want to prove that
$SP_{k+1}\not\to (H_m^k,c(\le))$; we proceed by contradiction.

Suppose there is a homomorphism $\varphi\colon SP_{k+1} \to (H_m^k,c(\le))$.
Since $(H_m^k, c(\le))$ is a vertex-transitive circularly ordered graph, we can
assume that $\varphi(v_1) = 0$. Let $u_i$ be the image of $v_i$ and note that
the only indices for which $u_i$ might be equal to $u_j$ are $i = 1$ and $j =
k+1$; the remaining pairs of vertices $u_i$ and $u_j$ must be different if $i
\ne j$. So there are $k$ different vertices $0 = u_1 < u_2 < \cdots < u_k <
km-1$ such that $u_iu_{i+1} \in E(H_m^k)$ for every $i \in \{1, \dots, k-1\}$,
and a vertex $u_{k+1} \in \{u_k+1, u_k+2, \dots, km-2, 0\}$ such that
$u_ku_{k+1} \in E(G)$. The existence of $u_{k+1}$ will yield the contradiction.
Recall that $H_m^k = K_{\sfrac{(km-1)}{m}}$, so there is an edge $rs \in
E(H_m^k)$ if and only if the circular distance between $r$ and $s$ is at least
$m$. Since $0 = u_1 < \dots <u_k$ is an increasing sequence in $\{0, 1, \dots,
km-1\}$, then $u_{i+1} - u_i \ge m$, so $u_k \ge (k-1)m$. Therefore, since
$u_{k+1} \in \{u_k+1, u_k+2, \cdots, km-2, 0\}$, then the circular distance
between $u_k$ and $u_{k+1}$ is at most the circular distance between $0$ and
$(k-1)m$ which is $km-1 - (k-1)m$ which equals $m-1$. Thus, $u_k$ and $u_{k+1}$
are not adjacent vertices, which contradicts the fact that $v_kv_{k+1}\in
E(SP_{k+1})$ and $\varphi(v_k) = u_k$ and $\varphi(v_{k+1}) = u_{k+1}$. Hence,
$SP_{k+1}\not \to (H_m^k, c(\le))$, so $H_m^k\in \mathcal{C}_{k+1}$.
\end{proof}

For every positive integer $k$, $k \ge 4$, we construct a set of linearly
ordered graphs, $\mathcal{PH}_k$ as follows. The \textit{straight cycle} on $k$
vertices $StC_k$, consists of the $k$-cycle, $v_1 \cdots v_k v_1$, where $v_i
\le v_j$ if and only if $i \le j$. The \textit{shifted straitgh path} on $k$
vertices, $sSt_k$, consists of the path on $k$ vertices, $v_1 \cdots v_k$, where
$v_k \le v_i$ for every $i\in\{1,\cdots, k\}$, and $v_i \le v_j$ for every $i,j
\in \{1, \dots, k-1\}$. Define $\mathcal{PH}_k$ as the set generated by all
linearly ordered spanning supergraphs of $\{St_k, sSt_k,StC_k, StC_{k-1}\}$.
In Figure~\ref{fig:PH} we depict these four generating linearly ordered graphs.

%%%%----------- TABLA CON F_6

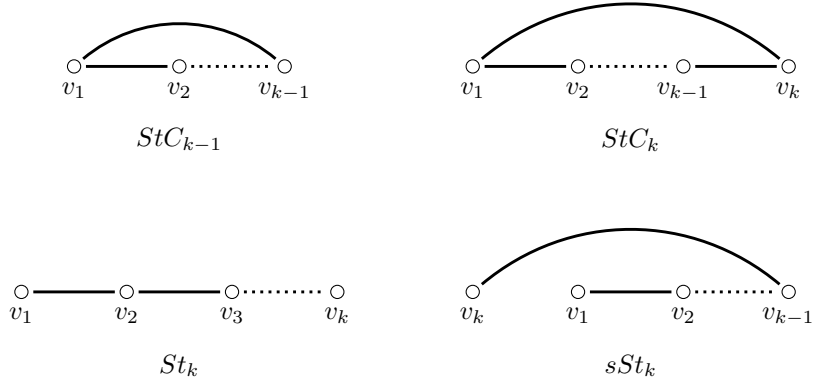
\begin{figure}[ht!]
\centering
\begin{tikzpicture}

\begin{scope}[xshift=-3cm, yshift = 3cm, scale=0.7]%StCk-1
\node [vertex, label=below:{$v_1$}] (1) at (-2,0){};
\node [vertex, label=below:{$v_2$}] (2) at (0,0){};
\node [vertex, label=below:{$v_{k-1}$}] (3) at (2,0){};
\node at (0,-1.4){$StC_{k-1}$};

\draw[edge]    (1)  edge  (2);
\draw[edge]    (2)  edge  [dotted] (3);
\draw[edge]    (3)  edge [bend right=40]   (1);
\end{scope}

\begin{scope}[xshift=3cm, yshift=3cm, scale=0.7]%StC_k
\node [vertex, label=below:{$v_1$}] (1) at (-3,0){};
\node [vertex, label=below:{$v_2$}] (2) at (-1,0){};
\node [vertex, label=below:{$v_{k-1}$}] (3) at (1,0){};
\node [vertex, label=below:{$v_k$}] (4) at (3,0){};
\node at (0,-1.4){$StC_{k}$};

\draw[edge]    (1)  edge  (2);
\draw[edge]    (3)  edge  (4);
\draw[edge]    (2)  edge  [dotted] (3);
\draw[edge]    (4)  edge [bend right=40]   (1);
\end{scope}

\begin{scope}[xshift=-3cm, yshift=0cm, scale=0.7]%StPk
\node [vertex, label=below:{$v_1$}] (1) at (-3,0){};
\node [vertex, label=below:{$v_2$}] (2) at (-1,0){};
\node [vertex, label=below:{$v_3$}] (3) at (1,0){};
\node [vertex, label=below:{$v_k$}] (4) at (3,0){};
\node at (0,-1.4){$St_k$};

\draw[edge]    (1)  edge  (2);
\draw[edge]    (2)  edge  (3);
\draw[edge]    (3)  edge   [dotted] (4);
\end{scope}

\begin{scope}[xshift=3cm, yshift = 0cm, scale=0.7]%sStPk
\node [vertex, label=below:{$v_k$}] (1) at (-3,0){};
\node [vertex, label=below:{$v_1$}] (2) at (-1,0){};
\node [vertex, label=below:{$v_2$}] (3) at (1,0){};
\node [vertex, label=below:{$v_{k-1}$}] (4) at (3,0){};
\node at (0,-1.4){$sSt_k$};

\draw[edge]    (4)  edge [bend right=40]  (1);
\draw[edge]    (2)  edge  (3);
\draw[edge]    (3)  edge   [dotted] (4);
\end{scope}

\end{tikzpicture}

\caption{The four generating linearly ordered graphs of $\mathcal{PH}_k$.}
\label{fig:PH}
\end{figure}
%%%%%%%%%%%%---------%%%%%%%%%------

\begin{proposition}
\label{prop:Hcirctolinear}
Let $G$ be a graph. If $G$ admits an $\mathcal{H}_k$-free circular ordering then
$G$ admits a $\mathcal{PH}_k$-free linear ordering.
\end{proposition}

\begin{proof}
Recall that $L$ is the function that ``linearizes'' a set of circularly ordered
graphs. Observation~\ref{obs:circtolinear} asserts that if $F$ is a set of
circularly ordered graphs that describes a property by forbidden circularly
ordered graphs, then $L(F)$ describes the same property by forbidden linearly
ordered graphs. So the statement of this proposition follows by this observation
and the fact that $\mathcal{PH}_k \subseteq L(\mathcal{H}_k)$.
\end{proof}

The following statement is a simple technical lemma that will be useful to prove
our main result of this section.

\begin{lemma}
\label{lem:technical}
Let $k$ be a positive integer, $k \ge 4$, let $G'$ be an acyclic oriented graph
with no directed path on $k$ arcs, and let $W$ be an oriented path in $G'$.
The following two assertions hold.
\begin{enumerate}
	\item If $|A(W)| = kn$ for some positive integer $n$ then, $|W^+| \le (k-1)n$
    and $|W^-| \ge n$.

	\item If the final arc of $W$ is a backward arc, then there is a positive
    integer $n$ such that $|W^+| \le (k-1)n$ and ${|W^-|} \ge n$ (regardless of
    $|A(W)|$).
\end{enumerate}
\end{lemma}

\begin{proof}
The first statement follows easily by induction on $n$. To prove the second
statement we will assume that $|A(W)| = km + l$ for some integer $l$, $0 \le l
\le k-1$. If $l = 0$ the second statement is a particular case of the first one.
Suppose $1 \le l \le k-1$ and let $W'$ be the subpath of $W$ obtained by
removing the final $l$ vertices. Clearly,  $|A(W')| = km$ and so by the first
statement $|W'^+| \le (k-1)m$ and $|W'^-| \ge m$. Since the final arc of $W$ is
a backward arc, then $|W^-| \ge m+1$ and $|W^+| \le (k-1)m + l \le (k-1)m + k-1
\le (k-1)(m+1)$. By letting $n = m+1$ the statement follows.
\end{proof}

\begin{theorem}
\label{thm:kcirc}
For any graph $G$ and a positive integer $k$, $k\ge 3$, the following statements
are equivalent:
\begin{itemize}
	\item $G$ admits a circular ordering $c(\le)$ such that $SP_{k+1} \not \to
    (G,c(\le))$,

	\item $G$ admits a $\mathcal{PH}_{k+1}$-free linear ordering, and

	\item the circular chromatic number of $G$ is strictly less than $k$.
\end{itemize}
\end{theorem}

\begin{proof}
Proposition~\ref{prop:Hcirctolinear} shows that the first item implies the
second one, while Proposition~\ref{prop:circhom} asserts that the third one
implies the first item. We now prove that the second statement implies the third
one. To do so, let $(G,\le)$ be a $\mathcal{PH}_{k+1}$-free linear ordering of a
graph $G$, and consider the orientation $G'$ of $G$ obtained by orienting every
edge $xy\in E(G)$ from $x$ to $y$ if $x\le y$. This is clearly an acyclic
orientation of $G$. We will show that  for every cycle $C$ of $G$ the strict
inequality $\frac{|C^+|}{|C^-|} < k-1$ holds, and thus, by
Theorem~\ref{thm:orientationsCirc} we conclude that $\chi_c(G) \le 1 +
\frac{|C^+|}{|C^-|} < k$.

Since the straight path on $k+1$ vertices belongs to $\mathcal{PH}_{k+1}$, then
$G'$ has no directed path on $k$ arcs. Let $C = c_1 \cdots c_mc_1$ be a cycle of
$G$ and without loss of generality assume that $c_1$ is the first vertex of $C$
with respect to $\le$. Moreover, we will assume that $m \ge k-1$; otherwise it
is immediate that $\frac{|C^+|}{|C^-|} < k-1$. We begin by first considering the
case when $c_m \le c_{m-1}$. In this case, $(c_m,c_{m-1}) \in A(G')$ and so the
oriented path $W = c_1Cc_m$ ends with a backward arc. Since $G'$ has no directed
path on $k$ arcs, by Lemma~\ref{lem:technical} there is an integer $l$ such that
$|W^+| \le (k-1)l$ and $|W^-| \ge l$. The only remaining arc in $C$ that does
not belong to $W$ is the arc $(c_1,c_m)$ and it is a backward arc in the
direction we are traversing $C$. Thus $|C^-| = |W^-|+1\ge l +1$ and $|C^+| =
|W^+| \le (k-1)l$, so $\frac{|C^+|}{|C^-|} \le \frac{(k-1)l}{(l+1)} < k-1$. Now
suppose that $c_{m-1} \le c_m$. Let $s$ be the maximum integer such that
$c_{m-i} \le c_{m-i+1}$ for every $i\in \{1, \dots, s\}$.

\begin{claim}\label{clm:1}
The strict inequality $s < k-1$ holds.
\end{claim}

Indeed, if $s \ge k-1$ then we have the following structure: $c_1 \le
c_{m-(k-1)} \le c_{m - (k-2)} \le \cdots \le c_m$ where $c_1c_m\in E(G)$ and
$c_{m-i}c_{m-i+1}\in E(G)$ for every $i\in \{1,\dots, k-1\}$. Regardless of
whether $c_1 = c_{m-(k-1)}$ or $c_1 \ne c_{m-(k-1)}$, we can find either $StC_k$
or $sSt_{k+1}$ as a linearly ordered subgraph of $(G,\le)$ which contradicts the
fact that $(G,\le)$ is $\mathcal{PH}_{k+1}$-free.  This concludes the proof of
Claim \ref{clm:1}.

By definition of $s$, we know that $c_{m-s} \le c_{m-(s+1)}$ so $(c_{m-s},
c_{m-(s+1)}) \in A(G')$. Hence, $(c_{m-s},c_{m-(s+1)})$ is a backward arc in the
oriented path $W = c_1Cc_{m-s}$. Again, by Lemma~\ref{lem:technical}, there is a
positive integer $l$ such that $|W^+|\le (k-1)l$ and $|W^-|\ge l$. Thus, $|C^+|
= |W^+| + s \le (k-1)l + s$ and $|C^-|  = |W^-| + 1 \ge l +1$, so
$\frac{|C^+|}{|C^-|} \le \frac{(k-1)l + s}{l+1}$. By Claim~\ref{clm:1} we know
that $s < k-1$ and therefore $\frac{|C^+|}{|C^-|} < \frac{(k-1)l + k-1}{l+1} =
k-1$. This shows that the orientation $G'$ satisfies that for every cycle $C$ of
$G$ the strict inequality $\frac{|C^+|}{|C^-|} < k-1$ holds. So by
Theorem~\ref{thm:orientationsCirc} we conclude that $\chi_c(G) \le 1 +
\frac{|C^+|}{|C^-|} < k$.
\end{proof}

Recall that $\mathcal{H}_{k+1}$ is a finite set of circularly ordered graphs
such that a graph belongs to $\mathcal{C}_k$ if and only if $G$ admits a
$\mathcal{H}_{k+1}$-free circular ordering.

\begin{corollary}
\label{cor:circularbound}
Let $k$ be a positive integer $k$, $k\ge 2$, and let $G$ be a graph. Then,
$\chi_c(G) < k$ if and only if $G$ admits an $\mathcal{H}_{k+1}$-free circular
ordering.
\end{corollary}

\section{Complexity issues}\label{sec:complexity}

Now we look at the problem of determining whether an input graph admits an
$F$-free circular ordering, where $F$ is a fixed finite set of circularly
ordered graphs.  We call this problem the \textit{$F$-free circular ordering
problem}.

Recall that for any set of linearly ordered patterns on three vertices there is
a polynomial time algorithm that determines whether an input graph admits an
$F$-free linear ordering of its vertices or not \cite{hellESA2014}. When it
comes to circular orderings this observation is trivial since for any graph on
three vertices $G$ there is a unique circular ordering of $G$. Thus, for any set
$F$ of circularly ordered graph on three vertices, the $F$-free circular
ordering problem can be solved in polynomial time. What about forbidding larger
circularly ordered graphs?  The little evidence gathered at this point, suggest
that if $F$ is a set of circularly ordered graphs on four vertices, then the
$F$-free circular ordering problem could be polynomial time solvable. Before
looking at this case, we show that there is a set of circularly ordered graphs
$F$ on five vertices such that the $F$-free circular ordering problem is
$NP$-complete. Thus, for any positive integer $k$, $k\ge 5$, there is a set $F$
of circularly ordered graphs on $k$ vertices such that the $F$-free circular
ordering problem is $NP$-complete.

In \cite{hatamiJGT47} Hatami and Tusserkani consider the following decision
problem. The input is a graph $G$ together with its chromatic number $k$, and
one must decide if $\chi_c(G) < k$. Their main result asserts that this problem
is $NP$-hard. By reading their proof one can notice that they actually show that
this problem is $NP$-hard even when restricted to $4$-chromatic graphs. For the
sake of clarity we state this result as follows.

\begin{proposition}\cite{hatamiJGT47}
\label{prop:4NPhard}
Given a $4$-chromatic graph $G$, the problem of determining if $\chi_c(G) < 4$
is $NP$-hard.
\end{proposition}

Now note that the problem stated in Proposition~\ref{prop:4NPhard} is also a
particular case of the problem of determining if an arbitrary graph $G$
satisfies $\chi_c(G) < 4$. Since this problem belongs to $NP$, then the
following statement directly follows from Proposition~\ref{prop:4NPhard}.

\begin{theorem}
\label{thm:4NPcomplete}
Given a graph $G$, the problem of determining if $\chi_c(G) < 4$ is
$NP$-complete.
\end{theorem}

Recall that by Corollary~\ref{cor:circularbound}, $\mathcal{H}_5$ is a set of
circularly ordered graphs such that a graph $G$ admits an $\mathcal{H}_5$-free
circular ordering if and only if $\chi_c(G) < 4$. Thus as a consequence of this
observation and Theorem~\ref{thm:4NPcomplete} we obtain the following corollary.

\begin{corollary}
\label{cor:5NPcomp}
For every positive integer $k$, $k\ge 5$, there is a set $F$ of circularly
ordered graphs on $k$ vertices such that the $F$-free circular ordering problem
is $NP$-complete.
\end{corollary}

\begin{proof}
By the arguments preceding this statement, $\mathcal{H}_5$ is a set of circular
ordered graphs on five vertices such that the $\mathcal{H}_5$-free circular
ordering problem is $NP$-complete. Moreover,  given a set $F$ of circular
ordered graphs on $k$-vertices such that the $F$-free circular ordering problem
is $NP$-complete, it is not hard to  construct a set $F'$ of circularly ordered
graphs on $(k+1)$-vertices such that the $F'$-free circular ordering is
$NP$-complete as well. Indeed, simply let $F'$ be the set of all circularly
ordered supergraphs on $(k+1)$-vertices of circularly ordered graphs in $F$.
\end{proof}

To conclude this section we construct a set $F_{CO}$ of circularly ordered
graphs such that the $F_{CO}$-free circular ordering problem nicely relates to
the \textit{cyclic ordering problem}. The cyclic ordering problem takes as an
input a set of ordered triples $R$ of some finite set $A$ and asks if the
triples of $R$ are generated by some circular ordering of $A$. This problem was
proved to be $NP$-complete in \cite{galilTCS5}.

Consider the graph $G_{aco}$ with vertex set $\{v_1,v_2,v_3,v_4,v_5,$ $v_6\}$
where $\{v_3,v_4,$ $v_5,v_6\}$ induce a clique and we add the edges $v_1v_6$,
$v_2v_5$ and $v_2v_6$. We define the circularly ordered graph $ACO$ as $G_{aco}$
with the circular closure of $v_1 < v_2 < v_3 < v_4 < v_5 <v_6$. We depict this
circularly ordered graph in Figure~\ref{fig:ACO}. Denote by $F_{CO}$ the set of
all circular orderings of $G_{aco}$ that are not isomorphic (as circularly
ordered graphs) to $ACO$. $F_{CO}$ is not an empty set, for instance, consider
$G$ with the circular ordering closure of $v_2 < v_1 < v_3 < v_4 < v_5 <v_6 $.

%%%%----------- Figura con ACAO

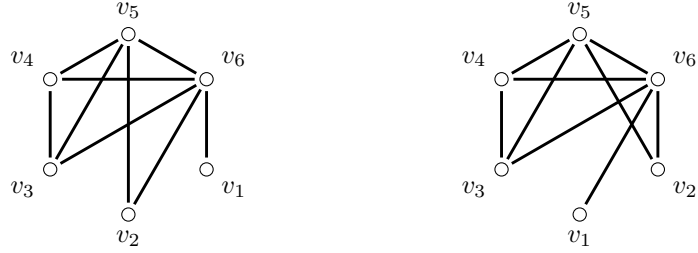
\begin{figure}[ht!]
\centering

\begin{tikzpicture}

\begin{scope}[xshift=-3cm, yshift = -4cm, scale=0.6]
\node [vertex, label=330:{$v_1$}] (1) at (330:2){};
\node [vertex, label=270:{$v_2$}] (2) at (270:2){};
\node [vertex, label=210:{$v_3$}] (3) at (210:2){};
\node [vertex, label=150:{$v_4$}] (4) at (150:2){};
\node [vertex, label=90:{$v_5$}] (5) at (90:2){};
\node [vertex, label=30:{$v_6$}] (6) at (30:2){};

\foreach \from/\to in {1/6,2/6,3/6,4/6,5/6,2/5,3/5,4/5,3/4}
\draw [edge] (\from) to (\to);
\end{scope}

\begin{scope}[xshift=3cm, yshift = -4cm, scale=0.6]
\node [vertex, label=270:{$v_1$}] (1) at (270:2){};
\node [vertex, label=330:{$v_2$}] (2) at (330:2){};
\node [vertex, label=210:{$v_3$}] (3) at (210:2){};
\node [vertex, label=150:{$v_4$}] (4) at (150:2){};
\node [vertex, label=90:{$v_5$}] (5) at (90:2){};
\node [vertex, label=30:{$v_6$}] (6) at (30:2){};

\foreach \from/\to in {1/6,2/6,3/6,4/6,5/6,2/5,3/5,4/5,3/4}
\draw [edge] (\from) to (\to);
\end{scope}

\end{tikzpicture}

\caption{On the left, a representation of $ACO$. On the right, a circular
ordered graph not isomorphic to $ACO$, but with the same underlying graph.}
\label{fig:ACO}
\end{figure}
%%%%%%%%%%%%---------%%%%%%%%%------

\begin{remark}
\label{rmk:ACO}
Note that there are only two automorphisms of $ACO$:  the identity, and the
transposition of $v_3$ with $v_4$ leaving every other vertex fixed. In any of
these two cases, the induced cyclic order in $\{ v_1, v_2, v_3 \}$ is the
circular closure of $v_1 < v_2 < v_3 $.
\end{remark}

\begin{theorem}
\label{thm:FCO-problem}
The $F_{CO}$-free circular ordering problem is $NP$-complete.
\end{theorem}

\begin{proof}
Consider an instance $(A,R)$ of the cyclic ordering problem. We construct the
input graph $G(A,R)$ for the $F_{CO}$-free circular ordering problem as follows.
The vertex set $V$ of $G(A,R)$, is the union $A \cup \mathfrak{R}$, where
$\mathfrak{R} = \{ r_x \colon r \in R \} \cup  \{r_y \colon r \in R \} \cup \{
r_z \colon r \in R \}$, i.e., $\mathfrak{R}$ contains three vertices for every
$r \in R$.  The set $A \subseteq V$ is an independent set and for every $r \in
R$ the vertices $r_x, r_y$ and $r_z$ induce a triangle. Finally, we add the
following edges between $A$ and $\mathfrak{R}$: for every $r = (a,b,c) \in R$ we
add the set of edges $\{r_za, r_zb, r_zc, r_yb, r_yc, r_xc\}$. In other words,
for $r \in R$, if $r = (a,b,c)$, the vertices $\{a, b, c, r_x, r_y, r_z\}$
induce a copy of $G_{ACO}$. A simple calculation shows that $|V| = |A| + 3|R|$,
and $|E| = 8|R|$, so this construction can be done in linear time. Before
showing that $G(A,R)$ is a yes-instance of the $F_{CO}$-free circular ordering
problem if and only if $(A,R)$ is yes-instance of the cyclic ordering problem,
we prove the following claim.

\begin{claim}
\label{cl:FCO-problem}
If $F$ is a set of vertices of $G(A,R)$ that induces a copy of $G_{ACO}$, then
there is an element $r \in R$, $r = (a,b,c)$, such that $F =
\{a,b,c,r_x,r_y,r_z\}$.
\end{claim}

We will show that if $F$ is such a set then $F = \{r_z\} \cup N(r_z)$ for some
$r\in R$; it should be clear that $N(r_z) = \{a,b,c,r_x,r_y\}$, where $r =
(a,b,c)$. Since $G_{ACO}$ has a unique universal vertex, then there is a unique
vertex $v_F\in F$ such that $F = \{v_F\} \cup N'$ where $N'\subseteq N(v_F)$. In
particular, the degree of  $v_F$ in $G(A,R)$ is at least five. It is not hard to
observe that for every $r\in R$ the degrees of $r_x$, $r_y$ and $r_z$ are $3$,
$4$ and $5$  respectively. Hence, $v_F = r_z$ for some $r\in R$ or $v_F \in A$.
By construction of $G(A,R)$, the graph induced by $\mathfrak{R}$ is a disjoint
union of triangles, so every connected subgraph of $G(A,R)[\mathfrak{R}]$
contains at most three vertices. Now note that the neighborhood of $v_F$ in
$G(A,R)[F]$ consists of an isolated vertex and a connected component on four
vertices. Since $A$ is an independent set, for every $a\in A$, $N(a) \subseteq
\mathfrak{R}$, so the neighborhood of $a$ cannot contain a connected component
on four vertices. Therefore $v_F \notin A$, and thus $v_F = r_z$ for some $r \in
R$. As mentioned at the begining of the paragraph, this concludes the proof of
Claim~\ref{cl:FCO-problem}.

Now we show that the proposed reduction translates yes-instances to
yes-instances and no-instances to no-instances of the corresponding problems.
Suppose that $G(A,R)$ admits an $F_{CO}$-free circular ordering $G'$.  Since the
only admissible circular orderings of $G_{ACO}$ are isomorphic to $ACO$, by
Remark~\ref{rmk:ACO}, for every ordered triple $(a,b,c) \in R$, the circular
ordering of $\{ a, b, c \}$ in $G'$ is $a < b < c < a$. So by considering the
circular ordering of $A$ inherited from $G'$, we obtain a circular ordering $R'$
of $A$, such that $R \subseteq R'$. Conversely, suppose that $(A,R)$ is a
yes-instance for the cyclic ordering problem, and let $a_1 < \cdots < a_n < a_1$
be the corresponding cyclic ordering of $A$. We want to extend this ordering to
a circular ordering of $V$. For every $r\in R$, $r = (a_i, a_j, a_k)$, include
$r_x,r_y,r_z$ in anyway such that $a_i < a_j < a_k < r_x < r_y < r_z < a_{k+1} <
a_i$. In other words, the circularly ordered graph induced by these vertices
($a_i,a_j,a_k,r_x,r_y$ and $r_z$) is a copy of $ACO$. Once we have extended the
circular ordering of $A$ to $V$ in this manner, call the resulting circularly
ordered graph $G'$. To see that $G'$ is an $F_{CO}$-free circularly ordered
graph, consider a set of six vertices $F \subseteq V$. If $F$ does not induce a
copy of $G_{ACO}$ in $G(A,R)$, then $F$ cannot induce a copy of any circularly
ordered graph in $F_{CO}$. If $F$ does induce a copy of $G_{ACO}$ in $G(A,R)$
then, by Claim~\ref{cl:FCO-problem}, there is an element $r\in R$, $r = (a, b,
c)$, such that $F = \{a, b, c, r_x, r_y, r_z\}$. Hence, by how we extended the
circular ordering of $A$ to $V$, $F$ induces a copy of $ACO$ in $G'$, and thus
it does not induce any circularly ordered graph of $F_{CO}$. Therefore, $G'$ is
an $F_{CO}$-free circular ordering of $G(A,R)$.

\end{proof}

\section{Conclusions and open problems}

It is now well-known that equipping a graph with a linear ordering of its vertex
set, or an orientation of its arc set, leads to characterizations of some
hereditary families of graphs in terms of finitely many forbidden induced
linearly ordered subgraphs or induced oriented subgraphs, respectively, in cases
where forbidding infinitely many induced subgraphs (without additional
structure) is needed.  In this work we show that similar results can be obtained
when we equip a graph with a circular ordering of its vertex set.   In this type
of problems, is natural to ask for the limitations of the proposed framework,
in particular, we think that the following problem is interesting.

\begin{problem}
Find a (relatively well-known) hereditary property that cannot be described by a
finite set of forbidden circularly ordered graphs.
\end{problem}

In this work we show that if a graph family can be described by finitely many
forbidden circularly ordered graphs, then it can be described by finitely many
forbidden linearly ordered graphs. So it is natural to ask if the converse
implication is also true. Since we do not think it is true, we propose this
question in a negative way.

\begin{question}
Is there a hereditary property described by finitely many forbidden linearly
ordered graphs that does not admit a characterization by finitely many
forbidden circularly ordered graphs?
\end{question}

In particular, we believe that the classes of $k$-colourable graphs are possible
candidates to answer the previous question in the negative, but finding any
such a class seems to be an interesting problem.

\begin{question}
For which positive integer $k$ the class of $k$-colourable graphs can be
described by finitely many forbidden circularly ordered graphs? In particular,
is there a finite set of circularly ordered graphs that describes  the class of
bipartite graphs?
\end{question}

There are a couple of characterizations of graphs with circular chromatic number
at least $3$ by certain unavoidable structures in every maximal triangle free
super graph. They read as follows.

\begin{theorem}
\cite{brandtCPC8}
Let $H$ be the graph obtained from the Petersen graph by deleting one vertex. A
graph $G$ has circular chromatic number at least $3$ if and only if every
maximal triangle-free supergraph $G_0$ of $G$ contains $H$ as a subgraph.
\end{theorem}

\begin{theorem}
\cite{pachDM37}
A graph $G$ has circular chromatic number at least $3$ if and only if every
maximal triangle-free supergraph $G_0$ of $G$ has an independent set whose
elements do not have a common neighbour.
\end{theorem}

Even though they are nice characterizations, they yield no information on how to
generalize it for larger circular chromatic numbers. Theorem~\ref{thm:kcirc}
provides a nice characterization through unavoidable structures in circular
orderings of graphs with circular chromatic number at least $k$ for any $k\ge
3$.

\begin{corollary}
Let $k$ be a positive integer, $k\ge 3$. A graph $G$ satisfies $\chi_c(G) \ge k$
if and only if for every circular ordering of $V(G)$, $v_1\le v_2\le \cdots \le
v_n \le v_1$, there are $k$ vertices $u_1 \le u_2 \le \cdots \le u_k$ such that
for every $i\in \{1,\cdots k-1\}$ there is an edge $u_iu_{i+1}\in E(G)$ (note
that $u_1$ might be $u_k$).
\end{corollary}

Theorem~\ref{thm:kcirc} also provides a characterization of the class of graphs
such that $\chi_C(G) = \chi(G)$ in terms of linear orderings of the vertex set.

\begin{theorem}
A graph $G$ satisfies that $\chi(G) = \chi_c(G) = k$ if and only if every
ordering $\le$ of $V(G)$ that avoids the structure $v_1 \le v_2 \le \cdots \le
v_{k+1}$ where  $v_iv_{i+1}\in E(G)$ for $i\in\{1,\cdots, k\}$, contains the
structure $u_1\le u_2 \le \cdots \le u_{k+1}$ where $u_1u_{k+1} \in E(G)$ and
$u_iu_{i+1}\in E(G)$ for $i\in\{2,\cdots, k\}$ (note that $u_1$ might be $u_2$).
\end{theorem}

\begin{corollary}
Consider a graph $G$ with $\chi(G) = k$. If $G$ admits a $k$-coloring, $c\colon
V(G)\to \{1,\cdots, k\}$, with no (possibly closed) walk $v_1v_2 \cdots v_{k+1}$
such that $c(v_i) = i$ for $i\in\{1,\cdots, k\}$ and $c(v_{k+1}) = 1$, then
$\chi_c(G) < \chi(G)$.
\end{corollary}

We dealt with complexity issues of the $F$-free circular ordering problem -- we
show that for every positive integer $k$, $k\ge 5$, there is a set $F$ of
circularly ordered graphs on $k$ vertices such that the $F$-free circular
ordering problem is $NP$-complete. Moreover, we already discussed that for every
set $F$ of circularly ordered  graphs on $3$ vertices, the $F$-free circular
ordering problem can be solved in polynomial time, so we ask:

\begin{question}\label{q:ord}
Is there a set $F$ of circularly ordered graphs on $4$ vertices such that
the $F$-free circular ordering problem is $NP$-complete?
\end{question}

We showed that the $\mathcal{H}_5$-free circular ordering problem is
$NP$-complete as a consequence of Theorem~\ref{thm:4NPcomplete}, which states
that determining if $\chi_c(G) < 4$ for an arbitrary graph $G$ is an
$NP$-complete problem as well. On the other hand, deciding if a graph $G$
satisfies that $\chi_c(G) < 2$ can be (trivially) done in polynomial time. It is
only natural to ask the following question.

\begin{question}
\label{q:circ}
Given a graph $G$, is the problem of determining if $\chi_c(G) < 3$ an
$NP$-complete problem?
\end{question}

Note that by Corollary~\ref{cor:circularbound}, Question~\ref{q:circ} is a
particular instance of Question~\ref{q:ord} since a graph $G$ admits an
$\mathcal{H}_4$-free circular ordering if and only if $\chi_c(G) < 3$, and
$\mathcal{H}_4$ consists of the triangle and circularly ordered graphs on $4$
vertices.

\end{document}